\documentclass[11pt]{article}
\usepackage[left=3.4cm,right=3.4cm,top=2.6cm,bottom=2.7cm]{geometry}
\usepackage{amssymb}
\usepackage{amsmath}
\usepackage{amsthm}
\usepackage{bm}
\usepackage{enumerate}
\usepackage{mathrsfs}
\usepackage{stmaryrd}
\usepackage{bbm}
\usepackage{graphicx}
\usepackage{caption}
\usepackage[utf8]{inputenc}
\usepackage{pgfplots}

\pgfplotsset{compat=1.14}
\pgfplotsset{every tick label/.append style={font=\small}}

\theoremstyle{definition}
\newtheorem{definition}{Definition}[section]
\newtheorem{notation}[definition]{Notation}
\newtheorem{example}[definition]{Example}
\newtheorem{assumption}[definition]{Assumption}

\theoremstyle{plain}
\newtheorem{proposition}[definition]{Proposition}
\newtheorem{theorem}[definition]{Theorem}
\newtheorem{lemma}[definition]{Lemma}
\newtheorem{corollary}[definition]{Corollary}

\theoremstyle{remark}
\newtheorem{remark}[definition]{Remark}

\numberwithin{equation}{section}

\newcommand{\E}{\mathbb{E}}
\newcommand{\N}{\mathbb{N}}
\renewcommand{\P}{\mathbb{P}}

\newcommand{\R}{\mathbb{R}}

\newcommand{\Y}{\mathbb{Y}}

\newcommand{\cA}{\mathcal{A}}
\newcommand{\cB}{\mathcal{B}}
\newcommand{\cC}{\mathcal{C}}
\newcommand{\cD}{\mathcal{D}}
\newcommand{\cE}{\mathcal{E}}
\newcommand{\cF}{\mathcal{F}}

\newcommand{\cH}{\mathcal{H}}

\newcommand{\cK}{\mathcal{K}}
\newcommand{\cL}{\mathcal{L}}
\newcommand{\cM}{\mathcal{M}}
\newcommand{\cN}{\mathcal{N}}
\newcommand{\cO}{\mathcal{O}}
\newcommand{\cP}{\mathcal{P}}

\newcommand{\cU}{\mathcal{U}}
\newcommand{\cV}{\mathcal{V}}
\newcommand{\cX}{\mathcal{X}}
\newcommand{\cY}{\mathcal{Y}}

\newcommand{\bY}{\mathbf{Y}}
\newcommand{\bZ}{\mathbf{Z}}

\newcommand{\sA}{\mathscr{A}}

\newcommand{\Ind}{\mathbf{1}}
\newcommand{\rd}{\mathrm{d}}
\newcommand{\vp}{\varphi}
\newcommand{\pa}{\partial}

\newcommand{\bme}{\bm{\eta}}
\newcommand{\frf}{\mathfrak{f}}

\newcommand{\deta}{\dot{\eta}}
\newcommand{\Lipb}{\textrm{\textnormal{Lip}}_b}
\newcommand{\ver}[1]{{\vert\kern-0.25ex\vert\kern-0.25ex\vert #1 \vert\kern-0.25ex\vert\kern-0.25ex\vert}}
\newcommand{\simplex}{\Delta_{[0,T]}}

\newcommand{\pvarst}{p,[s,t]}

\newcommand{\pvarzT}{p,[0,T]}
\newcommand{\pvarzt}{p,[0,t]}
\newcommand{\pvarrt}{p,[r,t]}

\newcommand{\ptvarzt}{\frac{p}{2},[0,t]}
\newcommand{\ptvarrt}{\frac{p}{2},[r,t]}

\newcommand{\onevarzt}{1,[0,t]}

\newcommand{\onevarrt}{1,[r,t]}
\newcommand{\Cpvar}{\cC^{p\textrm{\textnormal{-var}}}}
\newcommand{\Cptvar}{\cC^{\frac{p}{2}\textrm{\textnormal{-var}}}}

\newcommand{\rp}{\mathscr{V}^p}
\newcommand{\grp}{\mathscr{V}_g^{0\textrm{\textnormal{,}}p}}
\newcommand{\crp}{\cV^p_Y}
\newcommand{\crpt}{\cV^p_{\tilde{Y}}}
\newcommand{\tX}{\tilde{X}}
\newcommand{\tY}{\tilde{Y}}
\newcommand{\teY}{\tilde{\Y}}
\newcommand{\tbY}{\tilde{\bY}}
\newcommand{\tx}{\tilde{x}}
\newcommand{\ta}{\tilde{a}}
\renewcommand{\epsilon}{\varepsilon}
\newcommand{\tg}{\tilde{\gamma}}
\newcommand{\dg}{\dot{\gamma}}
\newcommand{\bu}{\bar{u}}

\DeclareMathOperator*{\esssup}{ess\,sup}

\DeclareMathOperator*{\argmin}{arg\,min}

\DeclareMathOperator*{\diag}{diag}

\begin{document}

%\title{Propagation of Uncertainty in Hidden Markov Models}
\title{Robust Filtering and Propagation of Uncertainty\\
in Hidden Markov Models}
\author{Andrew L. Allan\\
Department of Mathematics\\
ETH Z\"urich, Switzerland\\
andrew.allan\hspace{-0.6pt}{\fontfamily{ptm}\selectfont @}\hspace{-0.4pt}math.ethz.ch}
\date{March 14, 2021}
\maketitle

\vspace{-12pt}

\begin{abstract}
We consider the filtering of continuous-time finite-state hidden Markov models, where the rate and observation matrices depend on unknown time-dependent parameters, for which no prior or stochastic model is available. We quantify and analyze how the induced uncertainty may be propagated through time as we collect new observations, and used to simultaneously provide robust estimates of the hidden signal and to learn the unknown parameters, via techniques based on pathwise filtering and new results on the optimal control of rough differential equations.

\vspace{7pt}

\end{abstract}

\medskip

\noindent\textbf{Acknowledgment:} The author gratefully acknowledges financial support by the Swiss National Science Foundation via Project 200021\textunderscore 184647.

\section{Introduction}

The filtering of hidden processes from noisy observations is an important and routine problem arising in many applications. The basic problem is to derive an optimal online estimator for an unobserved `signal' process $X$ evolving randomly in time, from observations of another process $Y$ whose dynamics depend on the current state of the signal. Such stochastic filters have been derived and analyzed in various contexts, notably in the settings of linear underlying dynamics (the Kalman--Bucy filter \cite{Kalman1960,KalmanBucy1961}) and finite-state Markov processes (the Wonham filter \cite{Wonham1965}), but also in general nonlinear settings; see Bain and Crisan \cite{BainCrisan2009} or Crisan and Rozovskii \cite{CrisanRozovskii2011} for a comprehensive exposition of nonlinear filtering.

Stochastic filters take two primary inputs, namely a stochastic model for the underlying processes $X, Y$, and the observed data, coming from discrete observations of the path $t \mapsto Y_t$. The performance of a filter is naturally sensitive to the choice of stochastic model, and to the calibration of its parameters. In practice the parameters of the model are often unknown and may themselves vary in time. With the benefit of a stochastic model for the parameters (or simply a prior distribution in the case of constant parameters), one may in principle simply increase the dimension of the filter in order to simultaneously estimate the unknown parameters alongside the signal. In the absence of such a stochastic model, however, filtering alone is not sufficient, particularly when the parameters vary in time according to entirely unknown dynamics. Such general parameter uncertainty is thus a substantial and compelling problem.

The objective of the present work is to provide a theoretical framework to quantify and model how our uncertainty in the model parameters may be propagated through time, and to derive filters which provide robust estimates, both of the hidden signal process and of the unknown parameters.

We focus on the case of continuous-time finite-state hidden Markov models. The associated filters are known to be continuous with respect to their model parameters---corresponding convergence results are given in for example Chigansky and van Handel \cite{ChiganskyvanHandel2007} or Guo and Yin \cite{GuoYin2006}---but this by no means guarantees a satisfactory performance when the adopted parameters differ significantly from the true parameters. Uncertainty-robust filters for such systems were proposed by Borisov \cite{Borisov2008,Borisov2011} via minimax-filtering, whereby a best estimate is sought with respect to the worst case scenario, where `scenarios' here are represented by probability distributions over the space of all possible parameter values.

Such minimax procedures are by now classical, designed to find the estimate which minimizes the maximum expected loss over a range of plausible models, an approach which may be traced back at least as far as Wald \cite{Wald1945}, and has been applied in various settings, principally in those with linear underlying dynamics; see for example Martin and Mintz \cite{MartinMintz1983}, Miller and Pankov \cite{MillerPankov2005}, Siemenikhin \cite{Siemenikhin2016}, Siemenikhin, Lebedev and Platonov \cite{SiemenikhinLebedevPlatonov2005} or Verd\'u and Poor \cite{VerduPoor1984}. Invariably, however, by focusing exclusively on the worst case scenario, such procedures do not necessarily ensure a satisfactory performance under statistically realistic scenarios, and moreover make no attempt to learn the true parameter values, or more generally to evaluate our uncertainty and how it should be updated to reflect new observations.

Our approach is inspired by the discrete-time results of Cohen \cite{Cohen2020}, in which the data-driven robust (DR) expectation of \cite{Cohen2017} is introduced in a filtering context as a means of computing uncertainty-robust evaluations of functionals of the signal. In particular, in \cite{Cohen2020} it is shown that such DR-expectations actually provide the only way to construct an `expectation' which penalizes uncertainty, while preserving the natural properties of monotonicity, translation equivariance and constant triviality. In short, such nonlinear expectations consider evaluating random variables under a whole family of stochastic models, which itself is a standard approach to problems of robustness, except that the DR-expectation also penalizes such models according to how `unreasonable' they are considered to be. This penalisation is linked to statistical estimation of the models themselves, specifically via the corresponding negative log-likelihood function evaluated using the observed data.

Nonlinear expectations incorporating such model penalisation were first applied to continuous-time filtering in \cite{AllanCohen2019} in the context of Kalman--Bucy filtering---however, there the penalisation was based only on an initial calibration, and was not updated to incorporate new observations. In both \cite{Cohen2020} and \cite{AllanCohen2019}, the nonlinear expectation is seen to be characterised by what is essentially its convex dual, which, owing to the additive structure of the penalty, may be computed by a suitable dynamic programming principle.

In the present work we shall focus on this dual object, or `value function', as the object which describes the propagation of our uncertainty through time as we collect new observations. We adopt the classical setting in which the unobserved signal process $X$ is a continuous-time finite-state Markov chain, and study the uncertainty arising from both the unknown rate matrix of the chain $X$, and from the unknown observation matrix which determines the drift of the observation process $Y$. We will see how the value function in this setting may be formulated to encode our opinion of how `reasonable' both posterior distributions and parameter values are given our observations, and how it may then be used to compute robust estimates of each.

As alluded to above, filters are typically sensitive to both uncertainty of the model parameters, and to errors in the observed data due to imprecise modelling of the observation process. One may therefore desire a filter to be `robust' in both of these distinct senses, namely robust with respect to parameter uncertainty, and continuous with respect to the observation path $t \mapsto Y_t$ (in some suitable topology on path space). By taking a fully pathwise approach, by means of first `lifting' the observation process into the space of rough paths, we will see that our resulting filters are robust in both of these senses.

The first use of rough path theory in uncertainty-robust filtering was presented in Section~4 of Allan and Cohen \cite{AllanCohen2020}, allowing to extend the results of \cite{AllanCohen2019}, but remaining in the Kalman--Bucy setting. We highlight the results of Crisan, Diehl, Friz and Oberhauser \cite{CrisanDiehlFrizOberhauser2013} as the first use of rough paths to establish continuity of stochastic filters with respect to the (enhanced) observation path. Another early application of rough paths to robust statistics was exhibited by Diehl, Friz and Mai \cite{DiehlFrizMai2016}, who show that continuity of maximum likelihood estimators for diffusion processes with unknown parameters is recovered upon lifting the observation to a rough path. The collection of works \cite{AllanCohen2020,CrisanDiehlFrizOberhauser2013,DiehlFrizMai2016} thus provides three different but important notions of robustness in the statistics of observed diffusion processes.

We highlight that the penalisation approach we adopt, formulated in terms of the `reasonability' of parameter values, is inherently non-probabilistic, and thus does not require a prior distribution for the parameter values, nor any stochastic model for the dynamics of the parameters, which would be necessary for a purely Bayesian approach. In particular, our approach is suitable for cases in which filtering the parameter values is not feasible (or desirable due to additional nonlinearities in the underlying equations). The only input required from the user is their prior opinion of the reasonability of different parameter values, which in practice may simply consist of a prior estimate along with their confidence in this estimate.

One of the key steps in our approach is the derivation of a pathwise optimal control problem driven by the observation path $t \mapsto Y_t(\omega)$. As was demonstrated in Diehl, Friz and Gassiat \cite{DiehlFrizGassiat2017} and subsequently in \cite{AllanCohen2020}, rough path theory provides a convenient framework for pathwise control problems, in which they are formulated as the optimal control of a rough differential equation (RDE). In the current work we also provide new results in this direction, particularly on the regularity of the corresponding (rough) value function. A particular difficulty which we face, not encountered in \cite{DiehlFrizGassiat2017} or \cite{AllanCohen2020}, is that the spatial domain of our value function is actually time-dependent (and indeed rough), and thus requires delicate analysis. Moreover, and of independent interest, we also provide a new growth estimate for rough integrals depending on unconstrained parameters, which not only improves upon the corresponding estimate in \cite{AllanCohen2020}, but in fact turns out to be sharp (see Lemma~\ref{lemma sharp rough bound} below).

The structure of the paper is as follows. We begin in Section~\ref{sec rough path preliminaries} by recalling some basic concepts from rough path theory, and presenting our main results for controlled RDEs in Theorem~\ref{theoremHMMRDE}, the proof of which shall be postponed to the appendix. In Section~\ref{sec uncertain HMM} we shall introduce our underlying filtering problem and motivate our approach to quantifying uncertainty via model penalisation. This will then lead to an optimal control problem, and we shall establish properties of the corresponding value function in Section~\ref{sec pathwise control problem}. We will then show in Section~\ref{sec HJ equations} that the value function satisfies a suitable rough PDE. We shall present two simple numerical examples in Section~\ref{sec numerical examples}, and end with some brief concluding remarks in Section~\ref{sec conclusion}.

\section{Rough path preliminaries}\label{sec rough path preliminaries}

\subsection{Notation}

We consider a finite time interval $[0,T]$, and write $\simplex := \{(s,t) : 0 \leq s \leq t \leq T\}$ for the standard 2-simplex. For any path $X$ on $[0,T]$, we write the increment of $X$ over the interval $[s,t]$ as $X_{s,t} := X_t - X_s$, and write $\|X\|_\infty := \sup_{s \in [0,T]} |X_s|$ for the supremum norm. We also define the following function spaces. For given vector spaces $V$ and $W$, we write
\begin{itemize}
\item $\cL(V;W)$ for the space of linear maps from $V \to W$,
\item $\Lipb = \Lipb(V;W)$ for the space of bounded Lipschitz functions from $V \to W$,
\item $C^n_b = C^n_b(V;W)$ ($n \in \N$) for the space of $n$ times continuously differentiable (in the Fr\'echet sense) functions $\phi \colon V \to W$ such that $\phi$ and all its derivatives up to order $n$ are uniformly bounded,
\item $\Cpvar = \Cpvar([0,T];V)$ for the space of continuous paths $Y \colon [0,T] \to V$ with finite $p$-variation:
$$\|Y\|_p := \bigg(\sup_{\cP}\sum_{[s,t] \in \cP} |Y_{s,t}|^p\bigg)^{\hspace{-2.5pt}\frac{1}{p}} < \infty,$$
where the supremum is taken over all finite partitions $\cP$ of the interval $[0,T]$.
\end{itemize}

For $p \in [2,3)$ we write $\rp = \rp([0,T];\R^d)$ for the space of continuous $p$-rough paths, that is, pairs $\bY = (Y,\Y)$ such that $Y \colon [0,T] \to \R^d$ is a continuous path of finite $p$-variation, its `enhancement' $\Y \colon \simplex \to \R^{d \times d}$ is continuous and satisfies
$$\|\Y\|_{\frac{p}{2}} := \bigg(\sup_{\cP}\sum_{[s,t] \in \cP} |\Y_{s,t}|^{\frac{p}{2}}\bigg)^{\hspace{-2.5pt}\frac{2}{p}} < \infty,$$
and such that Chen's relation,
\begin{equation}\label{eq:Chensrelation}
\Y_{s,t} = \Y_{s,r} + \Y_{r,t} + Y_{s,r} \otimes Y_{r,t},
\end{equation}
holds for all times $s \leq r \leq t$, where $\otimes$ denotes the standard tensor product from $\R^d \times \R^d$ to $\R^d \otimes \R^d \cong \R^{d \times d}$.

\begin{remark}
As can be readily checked, any \emph{smooth} path $Y \colon [0,T] \to \R^d$ can be `lifted' in a canonical way to a rough path $\bY = (Y,\Y)$ by enhancing it with the integral
\begin{equation}\label{eq:iteratedintegrals}
\Y_{s,t} = \int_s^tY_{s,r}\otimes\rd Y_r,
\end{equation}
defined in the Riemann--Stieltjes sense. On the other hand, for a general path $Y$ of finite $p$-variation, the integral in \eqref{eq:iteratedintegrals} does not exist in the classical sense. The point here then is that the value of this integral is postulated by the enhancement $\Y$, which in practice is often constructed using stochastic integration.

For example, given a continuous semimartingale $Y$, for $p \in (2,3)$ one can construct an enhancement via Stratonovich integration:
\begin{equation}\label{eq:HMMStratlift}
\Y_{s,t} = \int_s^tY_{s,r}\otimes\circ\,\rd Y_r,
\end{equation}
that is, $\Y^{ij}_{s,t} = \int_s^tY^i_{s,r}\circ\rd Y^j_r$ for $i,j = 1,\ldots,d$, and the resulting lift $\bY = (Y,\Y)$ then defines a random rough path, so that $\bY(\omega) \in \rp$ for almost every $\omega \in \Omega$. Of course, when $i = j$ we have simply $\Y^{ii}_{s,t} = \int_s^tY^i_{s,r}\circ\rd Y^i_r = \frac{1}{2}(Y^i_{s,t})^2$. More generally, by integration by parts,
$$\Y^{ij}_{s,t} + \Y^{ji}_{s,t} = \int_s^tY^i_{s,r}\circ\rd Y^j_r + \int_s^tY^j_{s,r}\circ\rd Y^i_r = Y^i_{s,t}Y^j_{s,t}.$$
Thus, the additional information encoded by this lift is contained in the antisymmetric part of $\Y$, which corresponds to the L\'evy area of the process $Y$.
\end{remark}

For rough paths $\bY = (Y,\Y)$ and $\tbY = (\tY,\teY)$, we write
$$\ver{\bY}_p := \|Y\|_p + \|\Y\|_{\frac{p}{2}}$$
for the (inhomogeneous) rough path norm\footnote{Of course, $\rp$ is not actually a vector space due to the nonlinear structure of \eqref{eq:Chensrelation}.}, and
\begin{equation}\label{eq:rpdistance}
\|\bY;\tbY\|_p := \|Y - \tY\|_p + \|\Y - \teY\|_{\frac{p}{2}}
\end{equation}
for the corresponding rough path distance.

We will also consider the space of geometric rough paths $\grp \subset \rp$, defined as the closure of canonical lifts of smooth paths with respect to the pseudometric in \eqref{eq:rpdistance}. For example, when $Y$ is a semimartingale and we lift using Stratonovich integration, as in \eqref{eq:HMMStratlift}, the resulting lift turns out to be a (random) geometric rough path. This property of being well approximated by smooth paths allows one to make sense of solutions to a wide class of rough ODEs and PDEs---we will see an example of this in Definition~\ref{HMMdefnsolnroughHJB} below.

We will sometimes write e.g.~$\|Y\|_{\pvarst}$ for the $p$-variation of $Y$ over the subinterval $[s,t]$.

\subsection{Rough integration}

As noted above, the enhancement $\Y$ may be seen as postulating the value of the integrals $\Y^{ij}_{s,t} = \int_s^tY^i_{s,r}\,\rd Y^j_r$ for $i,j = 1,\ldots,d$. More generally, $\bY^j = (Y^j,(\Y^{ij})_{i=1}^d)$ contains all the information required to define integrals against $Y^j$, for any integrand in the space of paths `controlled' by $(Y^1,\ldots,Y^d)$ in the following sense.

Let $\bY = (Y,\Y)$ be a $p$-rough path for some $p \in [2,3)$. We say that a path $X \in \Cpvar([0,T];\R^m)$ is a controlled rough path (in the sense of Gubinelli \cite{Gubinelli2004}), if there exists a path $X' \in \Cpvar([0,T];\cL(\R^d;\R^m))$, known as the Gubinelli derivative of $X$ with respect to $Y$, such that the remainder term $R^X \colon \simplex \to \R^m$, defined implicitly by
\begin{equation}\label{eq:Gub deriv}
X_{s,t} = X'_sY_{s,t} + R^X_{s,t},
\end{equation}
satisfies $\|R^X\|_\frac{p}{2} < \infty$. We write $\crp = \crp([0,T];\R^m)$ for the space of controlled rough paths (with respect to $Y$), which becomes a Banach space when equipped with the norm $(X,X') \mapsto |X_0| + |X'_0| + \|X'\|_p + \|R^X\|_{\frac{p}{2}}$.

\begin{proposition}[Proposition~2.6 in \cite{FrizZhang2018}]\label{HMMproproughintegral}
Let $\bY = (Y,\Y) \in \rp([0,T];\R^d)$, and let $(X,X') \in \crp([0,T];\cL(\R^d;\R^m))$ be a controlled rough path. Then the limit
\begin{equation*}
\int_0^TX_r\,\rd\bY_r := \lim_{|\cP| \to 0}\sum_{[s,t] \in \cP} X_sY_{s,t} + X'_s\Y_{s,t}
\end{equation*}
exists\footnote{Strictly speaking, in making precise sense of the product $X'_s\Y_{s,t}$, we use the natural identification of $\cL(\R^d;\cL(\R^d;\R^m))$ with $\cL(\R^d \otimes \R^d;\R^m)$.}, where the limit is taken over any sequence of partitions $\cP$ of the interval $[0,T]$ such that the mesh size $|\cP| \to 0$. This limit (which does not depend on the choice of sequence of partitions) is called the rough integral of $X$ against $\bY$.
\end{proposition}

%Moreover, for any $0 \leq s < t \leq T$, we have the estimate
%\begin{equation}\label{eq:HMMroughintbound}
%\bigg|\int_s^tX_r\,\rd\bY_r - X_sY_{s,t} - X'_s\Y_{s,t}\bigg| \leq C_p\Big(\|R^X\|_{\ptvarst}\|Y\|_{\pvarst} + \|X'\|_{\pvarst}\|\Y\|_{\ptvarst}\Big)
%\end{equation}
%where the constant $C_p$ depends only on $p$.

\subsection{Rough differential equations}

We consider the rough differential equation (RDE) given by
\begin{equation}\label{eq:HMMRDE}
\rd X_t = b(X_t,\gamma_t)\,\rd t + \phi(X_t,\gamma_t)\,\rd\bY_t
\end{equation}
driven by a rough path $\bY \in \rp([0,T];\R^d)$, with some fixed $\gamma \in \Cptvar([0,T];\R^k)$.

\begin{theorem}\label{theoremHMMRDE}
Let $b \in \Lipb(\R^m \times \R^k;\R^m)$, $\phi \in C^3_b(\R^m \times \R^k;\cL(\R^d;\R^m))$, and $\psi \in C^3_b(\R^m \times \R^k;\cL(\R^d;\R^l))$. Let $p \in [2,3)$ and $T > 0$.
\begin{enumerate}[(i)]
\item For any rough path $\bY = (Y,\Y) \in \rp$, parameter $\gamma \in \Cptvar$ and any $x \in \R^m$, there exists a unique $X \in \Cpvar$ such that the controlled path $(X,X') = (X,\phi(X,\gamma)) \in \crp$ solves the RDE \eqref{eq:HMMRDE} driven by $\bY$ with parameter $\gamma$ and initial condition $X_0 = x$.
\item Suppose that $\ver{\bY}_{p,[0,T]} \leq L$ for some $L > 0$. We have the bound
\begin{equation}\label{eq:HMMroughintpsibound}
\bigg\|\int_0^\cdot\psi(X_r,\gamma_r)\,\rd\bY_r\bigg\|_{p,[0,T]} \leq C\Big(1 + T^{\frac{p - 1}{p}} + \|\gamma\|_{\frac{p}{2},[0,T]}^{\frac{p - 1}{2}}\Big)\ver{\bY}_{p,[0,T]},
\end{equation}
where the constant $C$ depends on $b, \phi, \psi, p$ and $L$.
\item Suppose that $\tX$ is the solution of \eqref{eq:HMMRDE} driven by $\tbY$ with parameter $\tg$, and suppose that $\ver{\bY}_{p,[0,T]}, \ver{\tbY}_{p,[0,T]} \leq L$. Assume also that $\|\gamma\|_{\frac{p}{2},[0,T]}, \|\tg\|_{\frac{p}{2},[0,T]} \leq M$ for some $M > 0$. Then we have the estimates
\begin{equation}
\|X - \tX\|_{p,[0,T]} \leq C'\Big(|X_0 - \tX_0| + \|\bY;\tbY\|_{p,[0,T]} + |\gamma_0 - \tg_0| + \|\gamma - \tg\|_{\frac{p}{2},[0,T]}\Big),\label{eq:RDEcontXglobal}
\end{equation}
\begin{align}
\bigg\|&\int_0^\cdot\psi(X_r,\gamma_r)\,\rd\bY_r - \int_0^\cdot\psi(\tX_r,\tg_r)\,\rd\tbY_r\bigg\|_{p,[0,T]}\nonumber\\
&\leq C''\Big(|X_0 - \tX_0| + \|\bY;\tbY\|_{p,[0,T]} + |\gamma_0 - \tg_0| + \|\gamma - \tg\|_{\frac{p}{2},[0,T]}\Big),\label{eq:RDEcontintpsiglobal}
\end{align}
where the constants $C', C''$ depend on $b, \phi, p, L$ and $M$, and $C''$ also depends on $\psi$.
\item Let $p \in (2,3)$. Suppose that $Y$ is a continuous semimartingale on some probability space $(\Omega,\cF,\P)$ and $\bY = (Y,\Y)$ is its Stratonovich lift, as in \eqref{eq:HMMStratlift}, so that $\bY(\omega) = (Y(\omega),\Y(\omega)) \in \grp$ for almost every $\omega \in \Omega$. Then the solution of the random RDE \eqref{eq:HMMRDE} is indistinguishable from the solution of the Stratonovich SDE
\begin{equation*}
\rd X_t = b(X_t,\gamma_t)\,\rd t + \phi(X_t,\gamma_t)\circ\rd Y_t,
\end{equation*}
and moreover the rough and Stratonovich integrals coincide almost surely, that is,
$$\int_0^t\psi(X_r(\omega),\gamma_r)\,\rd\bY_r(\omega) = \bigg(\int_0^t\psi(X_r,\gamma_r)\circ\rd Y_r\bigg)(\omega)$$
for almost every $\omega \in \Omega$.
\end{enumerate}
\end{theorem}

The proof of Theorem~\ref{theoremHMMRDE}, as well as those of Corollary~\ref{corollary reverse time} and Lemma~\ref{lemma sharp rough bound} below, are given in the appendix.

\begin{corollary}\label{corollary reverse time}
Recall the hypotheses of part (iii) of Theorem~\ref{theoremHMMRDE}. We also have
\begin{equation*}
\|X - \tX\|_{p,[0,T]} \leq C'\Big(|X_T - \tX_T| + \|\bY;\tbY\|_{p,[0,T]} + |\gamma_T - \tg_T| + \|\gamma - \tg\|_{\frac{p}{2},[0,T]}\Big),
\end{equation*}
\begin{align*}
\bigg\|&\int_0^\cdot\psi(X_r,\gamma_r)\,\rd\bY_r - \int_0^\cdot\psi(\tX_r,\tg_r)\,\rd\tbY_r\bigg\|_{p,[0,T]}\\
&\leq C''\Big(|X_T - \tX_T| + \|\bY;\tbY\|_{p,[0,T]} + |\gamma_T - \tg_T| + \|\gamma - \tg\|_{\frac{p}{2},[0,T]}\Big)
\end{align*}
(written in terms of the terminal values $X_T$, $\tX_T$, $\gamma_T$, $\tg_T$ instead of the initial values), with the same constants $C'$, $C''$ as before.
\end{corollary}

We have from part (ii) of Theorem~\ref{theoremHMMRDE} that the bound
\begin{equation}\label{eq:rough int bound sharp}
\bigg\|\int_0^\cdot\psi(X_r,\gamma_r)\,\rd\bY_r\bigg\|_{p,[0,T]} \leq C\big(1 + \|\gamma\|_{\frac{p}{2},[0,T]}^{q}\big)\ver{\bY}_{p,[0,T]}
\end{equation}
holds when $q = \frac{p - 1}{2}$, for some constant $C$ depending on $b, \phi, \psi, p, L$ and $T$. In fact, this estimate is sharp, in the sense of the following lemma.

\begin{lemma}\label{lemma sharp rough bound}
For any $q < \frac{p - 1}{2}$, there exist $b, \phi, \psi, p, L$ and $T$, such that there does \emph{not} exist a constant $C$ such that \eqref{eq:rough int bound sharp} holds for all $\gamma \in \Cptvar$ and $\bY \in \rp$ with $\ver{\bY}_{p,[0,T]} \leq L$.
\end{lemma}

\section{Uncertainty in hidden Markov models}\label{sec uncertain HMM}

\subsection{The Wonham filter}

Let $X$ be a continuous-time Markov chain taking values in the standard basis $\cX = \{e_1,\ldots,e_m\}$ of $\R^m$. We write $A_t = [a_{ij}(t)]_{m \times m}$ for the rate matrix\footnote{This is defined as the transpose of the $Q$-matrix of the chain, so that $\P(X_{t+\epsilon} = e_i\,|\, X_t = e_j) = \delta_{ij} + a_{ij}(t)\epsilon + o(\epsilon)$ as $\epsilon \to 0$, where $\delta_{ij}$ is the Kronecker delta.} of $X$, so that in particular the process $t \mapsto X_t - \int_0^t A_sX_s\,\rd s$ is a c\`adl\`ag martingale. We assume that $\E[X_0] = \pi_0$, for some element $\pi_0$ of the open probability simplex
$$S^m := \bigg\{x = (x_1,\ldots,x_m)^{\hspace{-1pt}\top} \in \R^m \ \bigg| \ 0 < x_j < 1 \ \ \forall j = 1,\ldots,m,\ \ \sum_{j=1}^m x_j = 1\bigg\}.$$
We consider the problem of estimating the current state of the chain $X$ from observations of the $\R^d$-valued process $Y = (Y^1,\ldots,Y^d)$, with $Y_0 = 0$ and dynamics
\begin{equation}\label{eq:HMM observation SDE}
\rd Y^i_t = (h_t^i)^{\hspace{-1pt}\top}\hspace{-1pt}X_t\,\rd t + \rd B^i_t, \qquad \quad i = 1,\ldots,d,
\end{equation}
where $h^i$ is an $\R^m$-valued time-dependent vector, and $B = (B^1,\ldots,B^d)$ is a standard $d$-dimensional Brownian motion. In the following we will write $h = (h^1,\ldots,h^d) \in \R^{m \times d}$ for the full observation matrix, and it will also be convenient to write
$$H^i_t = \diag(h^i_t),$$
i.e.~the diagonal matrix with diagonal elements given by the vector $h^i_t$.

\begin{remark}
It is worth pointing out that, since the signal process $X$ takes only finitely many values, there is no loss of generality in assuming a linear observation function in \eqref{eq:HMM observation SDE}, as any function on a finite set may be written in this form.
\end{remark}

We will denote by $(\cY_t)_{t \geq 0}$ the (completed) natural filtration generated by the observation process $Y$. The goal of the associated filtering problem is to determine, at each time $t$, the posterior distribution $\pi_t = \E[X_t\,|\,\cY_t]$. In the present setting, the filtering problem was resolved by Wonham \cite{Wonham1965}; see also Bain and Crisan \cite[Chapter~3]{BainCrisan2009}. The optimal filter $\pi$ is the unique continuous $\cY_t$-adapted solution of the stochastic differential equation
\begin{equation}\label{eq:filtereqnIto}
\rd\pi_s = A_s\pi_s\,\rd s + \sum_{i=1}^d \big(H^i_s - (h^i_s)^{\hspace{-1.5pt}\top}\hspace{-1pt}\pi_sI\big)\pi_s\big(\rd Y^i_s - (h^i_s)^{\hspace{-1.5pt}\top}\hspace{-1pt}\pi_s\,\rd s\big)
\end{equation}
where $I$ denotes the $m \times m$-identity matrix.

It will be convenient later to interpret the stochastic integral appearing in \eqref{eq:filtereqnIto} in the sense of Stratonovich, rather than that of It\^o. The Stratonovich version of \eqref{eq:filtereqnIto} is given by
\begin{equation}\label{eq:filtereqnStrat}
~\rd\pi_s = A_s\pi_s\,\rd s + \frac{1}{2}\sum_{i=1}^d \Big(\big((h^i_s)^{\hspace{-1.5pt}\top}\hspace{-1pt}H^i_s\pi_s\big)I - (H^i_s)^2\Big)\pi_s\,\rd s + \sum_{i=1}^d \big(H^i_s - (h^i_s)^{\hspace{-1.5pt}\top}\hspace{-1pt}\pi_sI\big)\pi_s\circ\rd Y^i_s.~
\end{equation}

\subsection{Parameter uncertainty}

We shall consider both uncertainty of the rate matrix $A$ of the Markov chain $X$, and uncertainty of the observation matrix $h$.

We recall the space of $m \times m$-rate matrices, that is, matrices $A = [a_{ij}]_{m \times m}$ such that $a_{ij} \geq 0$ for all $i \neq j$, and $\sum_{i = 1}^m a_{ij} = 0$ for all $j = 1,\ldots,m$.

\begin{notation}
We shall write $\cA$ for the set of admissible rate matrices, and $\cH$ for the set of admissible observation matrices, which we shall assume to be bounded connected subsets of $m \times m$-rate matrices and of $\R^{m \times d}$ respectively. Let $k \geq 1$ be the total dimension of the space of $\cA \times \cH$, which we note can be at most $m(m - 1) + md$.
\end{notation}

For notational clarity, we assume that the elements of the space $\cA \times \cH$ may be parameterised by the elements of $\R^k$. More precisely, we make the following assumption.

\begin{assumption}\label{assumption uncertainty parameterisation}
We assume that there exists a bijection from $\R^k \to \cA \times \cH$ which belongs to the class $C^3_b$.
\end{assumption}

\begin{example}
As a simple example, one might have $m = 2$, $d = 1$,
$$\cA = \bigg\{\bigg(\hspace{-4pt}\begin{array}{cc}
-\lambda & 1 - \lambda\\
\lambda & -(1 - \lambda)
\end{array}\hspace{-4pt}\bigg)\, \bigg|\ \lambda \in (0,1) \bigg\}, \qquad \text{and} \qquad \cH = \bigg\{\bigg(\hspace{-4pt}\begin{array}{c}
-\alpha\\
\alpha
\end{array}\hspace{-4pt}\bigg)\, \bigg|\ \alpha \in (1,2)\bigg\}.$$
In this case $k = 2$, and we may parametrise $\cA \times \cH$ by $\R^2$ via the mapping
$$\R^2 \ni (a_1,a_2) \longmapsto \bigg(\frac{1}{1 + e^{-a_1}}, \frac{2 + e^{-a_2}}{1 + e^{-a_2}}\bigg) = (\lambda,\alpha) \in (0,1) \times (1,2) \cong \cA \times \cH.$$
\end{example}

\begin{remark}
This framework includes cases in which the rate matrix $A$ (resp.~the observation matrix $h$) is known. In such cases we simply have $\cA = \{A^{\text{true}}\}$ where $A^{\text{true}}$ is the true rate matrix (resp.~$\cH = \{h^{\text{true}}\}$ where $h^{\text{true}}$ is the true observation matrix).
\end{remark}

We suppose that the true parameter is a Lipschitz continuous path taking values in $\R^k$. We shall therefore take as our uncertainty class the space $\sA$, where:

\begin{notation}
We denote by $\sA$ the space of Lipschitz continuous paths $\gamma \colon [0,T] \to \R^k$.
\end{notation}

%\begin{assumption}\label{assumptionzeroinR}
%We assume that $0 \in \cA$, i.e.~the zero matrix is admissible.
%\end{assumption}

\subsection{The setup}

Our setup is the following. For each choice of parameter $\gamma \in \sA$ we denote by $A = (A_t : t \in [0,T])$ and $h = (h_t : t \in [0,T])$ the rate matrix and observation matrix corresponding to $\gamma$ (via the bijection in Assumption~\ref{assumption uncertainty parameterisation}).

Let $X$ and $Y$ be two adapted processes on a filtered space $(\Omega,\cF,(\cF_t)_{t \in [0,T]})$. For each $\gamma \in \sA$ and initial distribution $\pi_0 \in S^m$, we let $\P^{\gamma,\pi_0}$ be a probability measure such that the law of $(X,Y)$ is equal to the law of $(\tilde{X}^{\gamma,\pi_0},\tilde{Y}^{\gamma,\pi_0})$, where, on some probability space, $\tilde{X}^{\gamma,\pi_0}$ is a Markov chain with transition matrix $A$ and initial distribution $\E[\tilde{X}^{\gamma,\pi_0}_0] = \pi_0$, and $\tilde{Y}^{\gamma,\pi_0}$ is a weak solution of \eqref{eq:HMM observation SDE}, i.e.~$\tilde{Y}^{\gamma,\pi_0}$ satisfies the SDE \eqref{eq:HMM observation SDE} driven by \emph{some} Brownian motion $B$ with the observation matrix $h$ and initial value $\tilde{Y}^{\gamma,\pi_0}_0 = 0$.

We note in particular that the processes $X$ and $Y$ as functions on $\Omega \times [0,T]$, and hence also the (uncompleted) filtration $\sigma(Y_s : s \in [0,t])$, $t \in [0,T]$, are defined independently of the choice of parameters---it is only the \emph{law} of $(X,Y)$ which varies depending on the choices of $\gamma$ and $\pi_0$.

To the family of all possible parameters $(\gamma,\pi_0) \in \sA \times S^m$, we can naturally associate the corresponding solution $\pi$ of the filtering equation (\eqref{eq:filtereqnIto} or \eqref{eq:filtereqnStrat}). Thus, at each time $t \geq 0$, we obtain in general a whole family of possible posterior distributions $\pi_t = x \in S^m$ for the signal, and a family of possible values $\gamma_t = a \in \R^k$ for the unknown parameter $\gamma$ at time $t$. Since we don't know which choice is the correct one, at each time $t$ we wish to know how to decide which posterior distribution $x \in S^m$ and parameter value $a \in \R^k$ is the most `reasonable' given our observations.

At each time $t > 0$, and for each choice of posterior distribution $x \in S^m$ and parameter value $a \in \R^k$, the central question we pose is thus the following:
\begin{quotation}
Given our observations, and given all the possible parameters choices, how reasonable is it that we would end up with the posterior distribution $\pi_t = x$ and parameter value $\gamma_t = a$ at time $t$?
\end{quotation}
To make this question more concrete, we need a notion of the `unreasonability' of different parameter choices. Mathematically, this notion may be represented by a `penalty' function which, at each time $t$, penalizes parameters according to how unreasonable we consider them to be given our observations up to time $t$.

Let us suppose for the moment that we have specified such a penalty function, denoted by $\beta_t(\gamma,\pi_0\,|\,\cY_t)$, which assigns a penalty to each choice of the parameters $\gamma, \pi_0$. Then, for a particular posterior distribution $x \in S^m$ and parameter value $a \in \R^k$, the most reasonable parameters $(\gamma,\pi_0)$ at time $t$ are those which attain the minimum of the set
$$\big\{\beta_t(\gamma,\pi_0\,|\,\cY_t)\ \big|\ (\gamma,\pi_0) \in \sA \times S^m \ \, \text{such that}\ \, (\pi_t,\gamma_t) = (x,a)\big\},$$
where $\pi = (\pi_s)_{s \in [0,t]}$ satisfies the filtering equation with rate and observation matrices corresponding to the parameter $\gamma$.

\subsection{The penalty function}

We suppose that our penalty takes the form of a negative log-posterior density. That is, we take
\begin{equation}\label{eq:penneglogpostHMM}
\beta_t(\gamma,\pi_0\,|\,\cY_t) = -\log\hspace{-1pt}\big(\vartheta_t(\gamma,\pi_0)L_t(\gamma,\pi_0\,|\,\cY_t)\big),
\end{equation}
where $\vartheta$ and $L(\hspace{1pt}\cdot\,|\,\cY_t)$ can be thought of as a prior and likelihood respectively.

\begin{remark}\label{remark HMM ignore const}
Since the posterior is only proportional to the product of prior and likelihood, \eqref{eq:penneglogpostHMM} is correct up to an additive constant. For simplicity we will omit this constant from our analysis, conceding that our penalty function is correct up to an additive constant. This constant may be reintroduced upon numerical computation, chosen to ensure that the penalty function always takes the value zero at its minimum.
\end{remark}

The penalty function in \eqref{eq:penneglogpostHMM} is built from the log-likelihood function, a familiar object from classical statistics. Penalties based on log-likelihoods form the basis of the data-driven robust (DR) expectation of \cite{Cohen2017}, which allows the level of penalisation of different parameter choices to be recursively updated through time as we collect new observations. Here we add to this an additional penalty based on our prior beliefs, which may be calibrated accordingly. We assume that the prior takes the form
\begin{equation}\label{eq:priortheta}
-\log\vartheta_t(\gamma,\pi_0) = \int_0^t\frf(\pi_s,\gamma_s,\dg_s)\,\rd s + g(\pi_0,\gamma_0),
\end{equation}
where as usual $\pi = (\pi_s)_{s \in [0,t]}$ is the posterior distribution corresponding to the parameters $\gamma$ and $\pi_0$, and $\dg$ is the derivative of $\gamma$. Here, the functions $\frf \colon S^m \times \R^k \times \R^k \to \R$ and $g \colon S^m \times \R^k \to \R$ may be calibrated to represent our prior beliefs about the plausibility of different parameter choices. In practice the function $\frf$ may also be time dependent, and may even depend on our observations provided that it is $\cY_t$-predictable.

By allowing $\frf$ to depend on the derivative $\dg$, we can penalize parameters, not only according to their value, but also according to how quickly they vary over time. For example, if we believe that the true parameter (or some component thereof) should remain fairly constant in time, then we can incorporate this belief by choosing the function $\frf$ to grow very quickly relative to the magnitude of $\dg_s$.

The natural choice for the likelihood $L_t(\hspace{1pt}\cdot\,|\,\cY_t)$ is the Radon--Nikodym derivative
$$L_t(\gamma,\pi_0\,|\,\cY_t) = \bigg(\frac{\rd\P^{\gamma,\pi_0}}{\rd\P^{\bar{\gamma},\bar{\pi}_0}}\bigg)_{\hspace{-3pt}\cY_t},$$
that is, the likelihood ratio of the (arbitrary) parameter choice $\gamma, \pi_0$, with respect to a (fixed) choice of reference parameters $\bar{\gamma}, \bar{\pi}_0$. We will now derive an explicit expression for this likelihood.

Recall (from e.g.~Bain and Crisan \cite[Chapter~2]{BainCrisan2009}) that for a given choice of parameters $\gamma, \pi_0$, the \emph{innovation} process $V = (V^1,\ldots,V^d)$, given in this setting by
$$\rd V^i_s = \rd Y^i_s - (h^i_s)^{\hspace{-1.5pt}\top}\hspace{-1pt}\pi_s\,\rd s,\qquad \quad i = 1,\ldots,d,$$
is a $\cY_t$-adapted Brownian motion under $\P^{\gamma,\pi_0}$, and moreover that, in this setting, $V$ generates the observation filtration (see Allinger and Mitter \cite{AllingerMitter1981}). Writing $\bar{\pi}$ (resp.~$\bar{V}$) for the posterior distribution (resp.~innovation process) under the reference measure $\P^{\bar{\gamma},\bar{\pi}_0}$, we have
$$\rd V^i_s = \rd \bar{V}^i_s - \big((h^i_s)^{\hspace{-1.5pt}\top}\hspace{-1pt}\pi_s - (\bar{h}^i_s)^{\hspace{-1.5pt}\top}\hspace{-1pt}\bar{\pi}_s\big)\rd s,\qquad \quad i = 1,\ldots,d.$$
Thus, by Girsanov's theorem (see e.g.~\cite[Chapter~15]{CohenElliott2015}), we can represent the likelihood as a stochastic exponential, namely
$$L_t(\gamma,\pi_0\,|\,\cY_t) = \exp\bigg(\sum_{i=1}^d \bigg(\big((h^i_s)^{\hspace{-1.5pt}\top}\hspace{-1pt}\pi_s - (\bar{h}^i_s)^{\hspace{-1.5pt}\top}\hspace{-1pt}\bar{\pi}_s\big)\rd \bar{V}^i_s - \frac{1}{2}\int_0^t\big|(h^i_s)^{\hspace{-1.5pt}\top}\hspace{-1pt}\pi_s - (\bar{h}^i_s)^{\hspace{-1.5pt}\top}\hspace{-1pt}\bar{\pi}_s\big|^{2}\,\rd s\bigg)\bigg).$$
Substituting $\rd\bar{V}^i_s = \rd Y^i_s - (\bar{h}^i_s)^{\hspace{-1.5pt}\top}\hspace{-1pt}\bar{\pi}_s\,\rd s$, a short calculation yields
\begin{align*}
-&\log L_t(\gamma,\pi_0\,|\,\cY_t)\\
&= \sum_{i=1}^d \bigg(-\int_0^t \big((h^i_s)^{\hspace{-1.5pt}\top}\hspace{-1pt}\pi_s - (\bar{h}^i_s)^{\hspace{-1.5pt}\top}\hspace{-1pt}\bar{\pi}_s\big)\rd Y^i_s + \frac{1}{2}\int_0^t\Big(\big|(h^i_s)^{\hspace{-1.5pt}\top}\hspace{-1pt}\pi_s\big|^2 - \big|(\bar{h}^i_s)^{\hspace{-1.5pt}\top}\hspace{-1pt}\bar{\pi}_s\big|^2\Big)\rd s\bigg).
\end{align*}

Since the reference parameters are taken to be fixed, they simply amount to an additive constant in the above expression. That is,
\begin{equation}\label{eq:negloglikeItoHMM}
-\log L_t(\gamma,\pi_0\,|\,\cY_t) = \sum_{i=1}^d \bigg(-\int_0^t(h^i_s)^{\hspace{-1.5pt}\top}\hspace{-1pt}\pi_s\,\rd Y^i_s + \frac{1}{2}\int_0^t\big|(h^i_s)^{\hspace{-1.5pt}\top}\hspace{-1pt}\pi_s\big|^2\,\rd s\bigg) + \text{const.}
\end{equation}
As in Remark~\ref{remark HMM ignore const}, we shall henceforth omit this constant.

For later convenience, we make the transformation
\begin{align*}
-\int_0^t(h^i_s)^{\hspace{-1.5pt}\top}\hspace{-1pt}\pi_s\,\rd Y^i_s &= -\int_0^t(h^i_s)^{\hspace{-1.5pt}\top}\hspace{-1pt}\pi_s\circ\rd Y^i_s + \frac{1}{2}\big\langle (h^i)^{\hspace{-1.5pt}\top}\hspace{-1pt}\pi, Y^i \big\rangle_{\hspace{-1pt}t}\\
&= -\int_0^t(h^i_s)^{\hspace{-1.5pt}\top}\hspace{-1pt}\pi_s\circ\rd Y^i_s + \frac{1}{2}\int_0^t(h^i_s)^{\hspace{-1.5pt}\top}\hspace{-1pt}\big(H^i_s - (h^i_s)^{\hspace{-1.5pt}\top}\hspace{-1pt}\pi_sI\big)\pi_s\,\rd s.
\end{align*}
Substituting back into \eqref{eq:negloglikeItoHMM}, we obtain
\begin{align}
-&\log L_t(\gamma,\pi_0\,|\,\cY_t)\nonumber\\
&= \sum_{i=1}^d \bigg(-\int_0^t(h^i_s)^{\hspace{-1.5pt}\top}\hspace{-1pt}\pi_s\circ\rd Y^i_s + \frac{1}{2}\int_0^t\big|(h^i_s)^{\hspace{-1.5pt}\top}\hspace{-1pt}\pi_s\big|^2 + (h^i_s)^{\hspace{-1.5pt}\top}\hspace{-1pt}\big(H^i_s - (h^i_s)^{\hspace{-1.5pt}\top}\hspace{-1pt}\pi_sI\big)\pi_s\,\rd s\bigg)\nonumber\\
&= \sum_{i=1}^d \bigg(-\int_0^t(h^i_s)^{\hspace{-1.5pt}\top}\hspace{-1pt}\pi_s\circ\rd Y^i_s + \frac{1}{2}\int_0^t(h^i_s)^{\hspace{-1.5pt}\top}\hspace{-1pt}H^i_s\pi_s\,\rd s\bigg).\label{eq:negloglikeStratHMM}
\end{align}
Substituting \eqref{eq:priortheta} and \eqref{eq:negloglikeStratHMM} into \eqref{eq:penneglogpostHMM}, we obtain
\begin{equation}\label{eq:beta expression}
\beta_t(\gamma,\pi_0\,|\,\cY_t) = \int_0^tf(\pi_s,\gamma_s,\dg_s)\,\rd s + \int_0^t\psi(\pi_s,\gamma_s)\circ\rd Y_s + g(\pi_0,\gamma_0),
\end{equation}
where, for notational simplicity, we have introduced the functions $f \colon S^m \times \R^k \times \R^k \to \R$ and $\psi \colon S^m \times \R^k \to \cL(\R^d;\R)$, defined by
\begin{equation}\label{eq:defn f psi}
f(\pi,\gamma,\dg) = \frf(\pi,\gamma,\dg) + \frac{1}{2}\sum_{i=1}^d (h^i)^{\hspace{-1.5pt}\top}\hspace{-1pt}H^i\pi, \quad \text{and} \quad \psi^i(\pi,\gamma) = -(h^i)^{\hspace{-1.5pt}\top}\hspace{-1pt}\pi,\quad i = 1,\ldots,d.
\end{equation}

\subsection{Pathwise filtering}

As discussed above, we propose to evaluate the unreasonableness of different posterior distributions $x \in S^m$ and parameter values $a \in \R^k$ by minimizing the penalty $\beta_t(\gamma,\pi_0\,|\,\cY_t)$ in \eqref{eq:beta expression} over all choices of the parameters $\gamma, \pi_0$ which would have resulted in the distribution $\pi_t = x$ and value $\gamma_t = a$.

Of course, in practice this optimization should depend on the particular realization of the observation process $Y$ that we actually observe. Thus, we do not wish to optimize the expectation of \eqref{eq:beta expression}, but rather we wish to simultaneously optimize with respect to each individual realization of the process $Y$. This motivates a pathwise interpretation of the filtering equation.

We proceed as follows. We first fix a reference measure $\P^{\bar{\gamma},\bar{\pi}_0}$. We then enhance the observation process $Y$ using Stratonovich integration:
\begin{equation}\label{eq:HMMfilterStratlift}
\Y_{s,t} := \int_s^tY_{s,r}\otimes\circ\,\rd Y_r \qquad \text{for all} \quad (s,t) \in \simplex,
\end{equation}
defined under the measure $\P^{\bar{\gamma},\bar{\pi}_0}$, which we recall defines a random geometric rough path $\bY = (Y,\Y) \in \grp$ for any $p \in (2,3)$.

Recall the Stratonovich filtering equation \eqref{eq:filtereqnStrat}. For notational simplicity, we rewrite this equation in the form
\begin{equation}\label{eq:filter Strat gamma}
\rd \pi_s = b(\pi_s,\gamma_s)\hspace{1pt}\rd s + \phi(\pi_s,\gamma_s) \circ \rd Y_s,
\end{equation}
where the coefficients $b \colon S^m \times \R^k \to \R^m$ and $\phi \colon S^m \times \R^k \to \cL(\R^d;\R^m)$ are chosen such as to make this equation consistent with \eqref{eq:filtereqnStrat}. Since the parameterisation $\R^k \to \cA \times \cH$ is of class $C^3_b$ (by Assumption~\ref{assumption uncertainty parameterisation}), we immediately also have that $b \in \Lipb$ and $\phi \in C^3_b$.

We note that in general the measures $\P^{\gamma,\pi_0}$ are not necessarily equivalent on $\cF_t$ (as different choices of the rate matrix $A$ may have different patterns of zero entries), and hence the \emph{completed} filtration $\cY_t = \sigma(Y_s : s \in [0,t]) \vee \cN$ may depend on the choice of $\gamma \in \sA$. However, since the integrand in \eqref{eq:HMMfilterStratlift} is trivially $\sigma(Y_s : s \in [0,t])$-adapted, the process $\Y$ coincides almost surely with the same integral defined under any other choice of measure $\P^{\gamma,\pi_0}$ (even though the corresponding completed filtrations $\cY_t$ may not agree).

Thus, defining $\pi$ as the solution of the RDE
\begin{equation}\label{eq:HMMfilterRDE}
\rd \pi_s = b(\pi_s,\gamma_s)\hspace{1pt}\rd s + \phi(\pi_s,\gamma_s) \, \rd \bY_s
\end{equation}
which exists by part (i) of Theorem~\ref{theoremHMMRDE}, then, for each choice of parameters $(\gamma,\pi_0)$, the corresponding solution $\pi$ of \eqref{eq:HMMfilterRDE} is indistinguishable from the solution of the Stratonovich equation \eqref{eq:filter Strat gamma} defined under $\P^{\gamma,\pi_0}$, and moreover the rough integral
$$\int_0^t \psi(\pi_s,\gamma_s) \,\rd \bY_s$$
coincides almost surely with the stochastic integral
$$\int_0^t \psi(\pi_s,\gamma_s) \circ \rd Y_s$$
(where we recall that $\psi \in C^3_b$ was defined in \eqref{eq:defn f psi}).

In particular, for each fixed (enhanced) realization $\bY_{[0,t]} = (Y|_{[0,t]}(\omega),\Y|_{\Delta_{[0,t]}}(\omega))$ of the observation process, we obtain an associated rough path $\bY \in \grp$, and we can write the penalty corresponding to this realization as
\begin{equation}\label{eq:beta pathwise}
\beta_t(\gamma,\pi_0\,|\,\bY_{[0,t]}) := \int_0^tf(\pi_s,\gamma_s,\dg_s)\,\rd s + \int_0^t\psi(\pi_s,\gamma_s)\,\rd \bY_s + g(\pi_0,\gamma_0).
\end{equation}

As discussed above, we propose to evaluate the reasonableness of posterior distributions $x \in S^m$ and parameter values $a \in \R^k$ by determining the most reasonable choice of the parameters $\gamma, \pi_0$ which would have resulted in the posterior $\pi_t = x$ and parameter value $\gamma_t = a$ at time $t$. The `unreasonableness' of each pair $(x,a) \in S^m \times \R^k$ is given by the functional $\kappa \colon [0,T] \times S^m \times \R^k \to \R$, defined by\footnote{Here and throughout, we adopt the convention that $\inf \emptyset = \infty$. For clarity we suppress the dependency of $\kappa$ on the observation path $\bY_{[0,t]}$ in our notation.}
\begin{equation}\label{eq:defnkappaHMM}
\kappa(t,x,a) = \inf\big\{\beta_t(\gamma,\pi_0\,|\,\bY_{[0,t]})\ \big|\ (\gamma,\pi_0) \in \sA \times S^m \hspace{6pt} \text{such that} \hspace{6pt} (\pi_t,\gamma_t) = (x,a)\big\},
\end{equation}
where the infimum is taken over all $\gamma \in \sA$ and $\pi_0 \in S^m$ such that $\gamma$ satisfies $\gamma_t = a$ and the solution $\pi$ of \eqref{eq:HMMfilterRDE} takes the terminal value $\pi_t = x$.

\subsection{Interpretation}

At each time $t$, the function $(x,a) \mapsto \kappa(t,x,a)$ encodes our opinion of how reasonable each posterior distribution $x \in S^m$ and each parameter value $a \in \R^k$ is at time $t$, given our observations. Thus, the map $t \mapsto \kappa(t,\cdot,\cdot)$ describes the propagation of our uncertainty through time.

Since $\kappa(t,x,a)$ measures the unreasonability of posteriors and parameter values, we obtain a filter which is robust to uncertainty by simply taking the minimum of $\kappa(t,\cdot,\cdot)$. That is, the most reasonable parameter values at each time are given by
\begin{equation}\label{eq:argmin a kappa}
t \, \longmapsto \, \argmin_{a \in \R^k} \inf_{x \in S^m} \kappa(t,x,a),
\end{equation}
and the most reasonable posteriors at each time are similarly given by
\begin{equation}\label{eq:argmin x kappa}
t \, \longmapsto \, \argmin_{x \in S^m} \inf_{a \in \R^k} \kappa(t,x,a).
\end{equation}

As noted above, the penalty function $\beta_t$ is defined up to an additive constant, which depends on time $t$ but does not depend on the choice of parameters $\gamma, \pi_0$. Similarly, the interpretation of the function $\kappa(t,\cdot,\cdot)$ is not affected by additive constants, and in practice it is therefore natural to shift the values of $\kappa$ so that $\inf_{(x,a)} \kappa(t,x,a) = 0$ for every $t \geq 0$. In particular, given a threshold $\lambda > 0$, one can then define a set of reasonable parameter values (or similarly posteriors) by setting
$$R^\lambda_t = \big\{a \in \R^k \ \big| \ \kappa(t,x,a) < \lambda \hspace{7pt} \text{for some} \hspace{7pt} x \in S^m\big\},$$
with the most reasonable parameter values, as in \eqref{eq:argmin a kappa}, being recovered as $\lim_{\lambda \to 0} R^\lambda_t$.

Analogously to Cohen \cite{Cohen2020}, one can also define an associated DR-expectation by setting
\begin{equation}\label{eq:nonlinexpHMM}
\cE(\vp(X_t)\,|\,\cY_t) = \esssup_{(\gamma,\pi_0) \in \sA \times S^m}\bigg\{\E^{\gamma,\pi_0}[\vp(X_t)\,|\,\cY_t] - \bigg(\frac{1}{k_1}\beta_t(\gamma,\pi_0\,|\,\cY_t)\bigg)^{\hspace{-2.5pt}k_2}\bigg\},
\end{equation}
defined for every functional $\vp \colon \cX \to \R$. Here, $k_1 > 0$ is an uncertainty aversion parameter, and the exponent $k_2 \geq 1$ is a shape parameter. As mentioned in the introduction, such expectations allow one to compute evaluations of random variables which penalize uncertainty, whilst retaining many of the natural properties one would expect from an expectation.

With $\pi$ defined as above, i.e.~as the solution of the rough filtering equation \eqref{eq:HMMfilterRDE}, it follows that
$$\sum_{j = 1}^m \pi_{j,t}\vp(e_j)$$
is a version of the conditional expectation $\E^{\gamma,\pi_0}[\vp(X_t)\,|\,\cY_t]$ for every choice of parameters and every functional $\vp$. Choosing this version, the nonlinear expectation in \eqref{eq:nonlinexpHMM} evaluated on a particular (enhanced) observation path $\bY_{[0,t]}$ is given by
\begin{align*}
\cE(\vp(X_t)\,|\,\bY_{[0,t]}&) = \sup_{(\gamma,\pi_0) \in \sA \times S^m} \bigg\{\sum_{j = 1}^m \pi_{j,t}\vp(e_j) - \bigg(\frac{1}{k_1} \beta_t(\gamma,\pi_0\,|\,\bY_{[0,t]})\bigg)^{\hspace{-2.5pt}k_2}\bigg\}.
\end{align*}

\begin{lemma}\label{lemmaDRintermskappa}
For any $\bY \in \grp$, $t \in [0,T]$ and functional $\vp \colon \cX \to \R$, we have that
\begin{equation}\label{eq:lem nonlinexp kappa}
\cE(\vp(X_t)\,|\,\bY_{[0,t]}) = \sup_{x \in S^m} \bigg\{\sum_{j=1}^m x_j\vp(e_j) - \bigg(\frac{1}{k_1}\inf_{a \in \R^k} \kappa(t,x,a)\bigg)^{\hspace{-2.5pt}k_2}\bigg\}
\end{equation}
where $\kappa$ is the function defined in \eqref{eq:defnkappaHMM}.
\end{lemma}

\begin{proof}
For any $(x,a) \in S^m \times \R^k$, we observe that
\begin{align*}
&\sup \bigg\{\sum_{j = 1}^m \pi_{j,t}\vp(e_j) - \bigg(\frac{1}{k_1} \beta_t(\gamma,\pi_0\,|\,\bY_{[0,t]}\bigg)^{\hspace{-2.5pt}k_2}\, \bigg|\, \genfrac{}{}{0pt}{}{(\gamma,\pi_0) \in \sA \times S^m \hspace{6pt} \text{such that}}{(\pi_t,\gamma_t) = (x,a)}\bigg\}\\
&= \sum_{j=1}^m x_j\vp(e_j) - \bigg(\frac{1}{k_1}\kappa(t,x,a)\bigg)^{\hspace{-2.5pt}k_2}.
\end{align*}
We then obtain \eqref{eq:lem nonlinexp kappa} upon taking the supremum over $(x,a) \in S^m \times \R^k$.
\end{proof}

Assuming that the function $\kappa(t,\cdot,\cdot)$ has a unique minimum point, and that it has been shifted if necessary so that it takes the value zero at this point, it follows from Lemma~\ref{lemmaDRintermskappa} that
\begin{equation*}
\cE(\vp(X_t)\,|\,\bY_{[0,t]}) \, \longrightarrow \, \sum_{j=1}^m \hat{x}_{j,t}\vp(e_j) \qquad \text{as} \quad k_1 \, \longrightarrow \, 0,
\end{equation*}
where $\hat{x}_t$ is the most reasonable posterior value at time $t$, as in \eqref{eq:argmin x kappa}.

\medskip

The key insight of our approach is that the function $\kappa$, as defined in \eqref{eq:defnkappaHMM}, has the form of the value function of an optimal control problem. To make this precise we introduce:
\begin{notation}
We denote by $\cU$ the space of bounded measurable paths $u \colon [0,T] \to \R^k$.
\end{notation}

Recalling \eqref{eq:beta pathwise}, we rewrite \eqref{eq:defnkappaHMM} as
\begin{align}
&\kappa(t,x,a)\label{eq:kappainfA}\\
&= \inf_{u \in \cU} \bigg\{\int_0^tf(\pi^{t,x,a,u}_s,\gamma^{t,a,u}_s,u_s)\,\rd s + \int_0^t\psi(\pi^{t,x,a,u}_s,\gamma^{t,a,u}_s)\,\rd \bY_s + g(\pi^{t,x,a,u}_0,\gamma^{t,a,u}_0)\bigg\},\nonumber
\end{align}
where we now interpret $\dg^{t,a,u} = u \in \cU$ as a control, and $\pi^{t,x,a,u}, \gamma^{t,a,u}$ as the state variables, which satisfy the controlled dynamics given by
\begin{align*}
\rd \pi^{t,x,a,u}_s &= b(\pi^{t,x,a,u}_s,\gamma^{t,a,u}_s)\hspace{1pt}\rd s + \phi(\pi^{t,x,a,u}_s,\gamma^{t,a,u}_s) \, \rd \bY_s, & \pi^{t,x,a,u}_t &= x,\\
\rd \gamma^{t,a,u}_s &= u_s\,\rd s, & \gamma^{t,a,u}_t &= a.
\end{align*}

One should take care here to exclude unphysical trajectories. That is, given a path $\bY$ and terminal condition $(\pi_t,\gamma_t) = (x,a)$, there may exist choices of control $u \in \cU$ for which the solution $\pi = (\pi_s)_{s \in [0,t]}$ of \eqref{eq:HMMfilterRDE} leaves the domain $S^m$. Such tuples $(t,x,a,u)$ do not correspond to a physical initial value $\pi_0$, and should thus be discarded.

\begin{example}
Let us for example suppose that $m = 2$, $d = 1$, and consider the rate and observation matrices where
$$A_t = \bigg(\hspace{-4pt}\begin{array}{cc}
-\lambda & \mu\\
\lambda & -\mu
\end{array}\hspace{-4pt}\bigg), \qquad h_t = \bigg(\hspace{-4pt}\begin{array}{c}
-\alpha\\
\alpha
\end{array}\hspace{-4pt}\bigg) \qquad \text{for all} \quad t \in [0,T],$$
for some $\lambda, \mu, \alpha > 0$. In this case the second component $\pi_2$ of the filter $\pi = (\pi_1,\pi_2)^{\hspace{-1pt}\top}$ satisfies
$$\rd\pi_2 = \big(\lambda(1 - \pi_2) - \mu\pi_2\big)\rd s + 2\alpha\pi_2(1 - \pi_2)\hspace{1pt}\rd \bY_s.$$
Choosing a terminal condition close to the boundary so that $\pi_t = x \simeq (1,0)^{\hspace{-1pt}\top}$, we then have
$$\rd\pi_{2,s} \simeq \lambda\,\rd s, \qquad \text{with} \quad \pi_{2,t} \simeq 0.$$
Since $\lambda > 0$, it is then inevitable that $\pi_{2,s} < 0$ for some $s < t$, and hence that $\pi_s \notin S^2$.
\end{example}

Although we don't obtain an initial value $\pi_0 \in S^m$ for trajectories which leave the domain, we can simply assign an infinite initial cost to all such trajectories, so that the corresponding controls are never considered when taking the infimum in \eqref{eq:kappainfA}.

Moreover, in general there will exist terminal conditions $(t,x,a)$ for which \emph{every} choice of control $u \in \cU$ results in a trajectory $\pi = (\pi_s)_{s \in [0,t]}$ which leaves the domain $S^m$. These are posteriors $x \in S^m$ which do not correspond to any pair of parameters $(\gamma,\pi_0) \in \sA \times S^m$, and which are therefore totally implausible given our observations up to the current time. In these cases we have simply $\kappa(t,x,a) = \infty$.

\begin{notation}\label{notation Qt cD}
Let $t \in [0,T]$ and $a \in \R^k$. We denote the set of `plausible' posteriors at time $t$ (for which at least one physical trajectory exists) by
\begin{equation*}
Q_t := \big\{x \in S^m \ \big| \ \kappa(t,x,a) < \infty\big\} = \big\{x \in S^m \ \big| \ \exists u \in \cU \hspace{7pt} \text{such that} \hspace{7pt} \pi^{t,x,a,u}_0 \in S^m\big\}.
\end{equation*}
Since we impose no uniform bound on the controls $u \in \cU$, it is easy to deduce that if $\kappa(t,x,a) < \infty$ for some $a \in \R^k$, then in fact $\kappa(t,x,a) < \infty$ for \emph{all} $a \in \R^k$. Thus, the set $Q_t$ defined above is independent of the choice of $a \in \R^k$. We will also denote by
\begin{equation}\label{eq:defn domain D}
\cD := \bigcup_{t \in [0,T]} (\{t\} \times Q_t \times \R^k)
\end{equation}
the domain on which $\kappa$ is finite.
\end{notation}

It is clear that $Q_0 = S^m$. Moreover, the domain $Q_t$, which is easily seen to be an open subset of $S^m$, does not depend on the choice of the functions $f$ and $g$, but it does depend on the space $\cA \times \cH$ and on the realization of the observation path $\bY_{[0,t]}$. The boundary $t \mapsto \pa Q_t$ therefore also inherits the roughness of $\bY$.

\begin{remark}
The fact that the set of plausible posteriors $Q_t$ is in general a proper subset of $S^m$ should not be too surprising. In particular, in the degenerate case with no uncertainty, so that $\cA \times \cH$ is just the singleton $\{(A^{\text{true}},h^{\text{true}})\}$ and the initial distribution $\pi_0 = \pi_0^{\text{true}}$ is known, the set $Q_t$ reduces to the singleton $\{\pi^{\text{true}}_t\}$, where $\pi^{\text{true}}$ is the filter corresponding to the true parameters. Moreover, we cannot expect all posteriors to be plausible (i.e.~reachable by at least one filter trajectory) without an assumption of irreducibility on the admissible rate matrices.
\end{remark}

\begin{remark}
Although in general the domain $Q_t$ is a proper subset of $S^m$, there are cases in which $Q_t = S^m$ (so that all posteriors $x \in S^m$ are considered to be plausible) at every time $t$. We will see an example of this in Section~\ref{subsec uncertain rate matrix}.
\end{remark}

\section{An unbounded pathwise control problem}\label{sec pathwise control problem}

In the previous section we formulated an optimal control problem, which for convenience we restate here. We have the value function
\begin{align}
&\kappa(t,x,a)\label{eq:valuefuncHMM}\\
&= \inf_{u \in \cU} \bigg\{\int_0^t f(\pi^{t,x,a,u}_s,\gamma^{t,a,u}_s,u_s)\,\rd s + \int_0^t \psi(\pi^{t,x,a,u}_s,\gamma^{t,a,u}_s)\,\rd \bY_s + g(\pi^{t,x,a,u}_0,\gamma^{t,a,u}_0)\bigg\},\nonumber
\end{align}
for $(t,x,a) \in \cD$ (as defined in \eqref{eq:defn domain D}), where the state variables $\pi^{t,x,a,u}, \gamma^{t,a,u}$ satisfy the controlled dynamics:
\begin{align}
\rd \pi^{t,x,a,u}_s &= b(\pi^{t,x,a,u}_s,\gamma^{t,a,u}_s)\hspace{1pt}\rd s + \phi(\pi^{t,x,a,u}_s,\gamma^{t,a,u}_s) \, \rd \bY_s, & \pi^{t,x,a,u}_t &= x,\label{eq:contrdynamicsHMM}\\
\rd \gamma^{t,a,u}_s &= u_s\,\rd s, & \gamma^{t,a,u}_t &= a.\nonumber
\end{align}
We recall that, under Assumption~\ref{assumption uncertainty parameterisation}, we have that $b \in \Lipb$ and $\phi \in C^3_b$, so that the RDE \eqref{eq:contrdynamicsHMM} has a unique solution by part~(i) of Theorem~\ref{theoremHMMRDE}.

As discussed in the previous section, in \eqref{eq:valuefuncHMM} we assign an infinite initial cost to all trajectories $\pi^{t,x,a,u}$ which leave the domain $S^m$ at any time $s < t$. We will sometimes omit the superscripts on the state variables when no confusion is likely to occur.

\subsection{Observations and assumptions}

We begin with some observations. First, we note that this is a `backward' control problem, in the sense that we prescribe a terminal condition $(\pi_t,\gamma_t) = (x,a)$ for the state trajectories and, for each choice of control $u$, solve the controlled dynamics backwards in time to obtain the corresponding initial values $(\pi_0,\gamma_0)$.

More significantly, here we wish to perform the optimization for every fixed (enhanced) realization $\bY_{[0,T]}$ of the stochastic process $Y$. This type of problem is known as `pathwise stochastic control'. In fact, we have formulated our problem in terms of the optimal control of a rough differential equation, which we wish to perform for an arbitrary geometric rough path $\bY \in \grp$. Control problems of this type were first studied by Diehl et al.~\cite{DiehlFrizGassiat2017}, and subsequently by Allan and Cohen \cite{AllanCohen2020}.

Moreover, the control problem stated above is unbounded, in the sense that, as we will see, the value function $\kappa(t,x,a)$ `blows up' for values of $x$ which are close to the boundary of $Q_t$, and also for very large values of $a$. This is because such values of $x$ and $a$ are considered to be very unreasonable, and are thus assigned a very large cost.

\begin{notation}
Writing $\overline{S}^m$ for the closure of $S^m$ in $\R^m$, we denote by $\pa S^m := \overline{S}^m \setminus S^m$ the boundary of the domain $S^m$. We then denote by
$$d(x,\pa S^m) := \inf_{y \in \pa S^m} |x - y|$$
the distance of a point $x \in S^m$ to the boundary $\pa S^m$.

We write $C^\uparrow(S^m \times \R^k;\R)$ for the space of continuous functions $\tilde{g} \colon S^m \times \R^k \to \R$ which explode near the boundary of the domain $S^m$, and also for extreme values of $\R^k$, that is, functions such that
\begin{align*}
\inf_{a \in \R^k} \tilde{g}(x,a) \, \longrightarrow \, \infty \qquad &\text{as} \quad d(x,\pa S^m) \, \longrightarrow \, 0, \quad \text{and}\\
\inf_{x \in S^m} \tilde{g}(x,a) \, \longrightarrow \, \infty \qquad &\text{as} \quad |a| \, \longrightarrow \, \infty.
\end{align*}
\end{notation}

\begin{assumption}\label{assumptions f g}
We assume that
\begin{enumerate}[(i)]
\item the running cost $f = f(x,a,u)$ and the initial cost $g = g(x,a)$ are continuous, bounded below, and locally Lipschitz continuous in $(x,a)$ uniformly in $u$,
\item the running cost is superlinear in $u$, in the sense that
\begin{equation*}
\inf_{(x,a) \in S^m \times \R^k} \frac{f(x,a,u)}{|u|} \, \longrightarrow \, \infty \qquad \text{as} \quad |u| \, \longrightarrow \, \infty,
\end{equation*}
\item and that $g \in C^\uparrow(S^m \times \R^k;\R)$.
\end{enumerate}
\end{assumption}

\begin{notation}\label{notation lesssim}
Let $L > 0$ such that $\ver{\bY}_{p,[0,T]} \leq L$. In the following we shall use $\lesssim$ to denote inequality up to a multiplicative constant which may depend on any of the dimensions $m, d, k$, the functions $b, \phi, \psi, f, g$, the measure of regularity $p \in (2,3)$, the bound $L$, and the terminal time $T$.
\end{notation}

\begin{lemma}\label{lemma value func loc bdd}
Under Assumptions~\ref{assumption uncertainty parameterisation} and \ref{assumptions f g}, the value function $\kappa$ is bounded below, and locally bounded above on $\cD$.
\end{lemma}

\begin{proof}
\textbf{Step 1.}
By part (ii) of Theorem~\ref{theoremHMMRDE}, we have
\begin{equation*}
\bigg|\int_0^t\psi(\pi^{t,x,a,u}_s,\gamma^{t,a,u}_s)\,\rd\bY_s\bigg| \leq \bigg\|\int_0^\cdot\psi(\pi^{t,x,a,u}_s,\gamma^{t,a,u}_s)\,\rd\bY_s\bigg\|_{p,[0,t]} \lesssim 1 + \|\gamma^{t,a,u}\|_{\ptvarzt}^{\frac{p - 1}{2}}.
\end{equation*}
We then note that
\begin{equation}\label{eq:bound gamma int u}
\|\gamma^{t,a,u}\|_{\ptvarzt}^{\frac{p - 1}{2}} \leq \|\gamma^{t,a,u}\|_{\onevarzt}^{\frac{p - 1}{2}} = \bigg(\int_0^t |u_s|\,\rd s\bigg)^{\hspace{-2.5pt}\frac{p - 1}{2}} \lesssim \int_0^t |u_s|^{\frac{p - 1}{2}}\,\rd s.
\end{equation}
By part (ii) of Assumption~\ref{assumptions f g}, for any $\varepsilon > 0$ we infer the existence of a constant $C_\epsilon$ such that
\begin{equation}\label{eq:Afepsilonbound}
|u|^{\frac{p - 1}{2}} \leq C_\epsilon + \epsilon f(x,a,u)
\end{equation}
for all $(x,a,u) \in S^m \times \R^k \times \R^k$. Choosing $\epsilon$ sufficiently small, we deduce that
\begin{equation}\label{eq:HMMpsiintlocbound}
\bigg|\int_0^t\psi(\pi^{t,x,a,u}_s,\gamma^{t,a,u}_s)\,\rd\bY_s\bigg| \leq C + \frac{1}{2}\int_0^t f(\pi^{t,x,a,u}_s,\gamma^{t,a,u}_s,u_s)\,\rd s
\end{equation}
for some new constant $C$. Thus, recalling \eqref{eq:valuefuncHMM}, we have
\begin{equation}\label{eq:kappaboundedbelow}
\kappa(t,x,a) \geq \inf_{u \in \cU} \bigg\{\frac{1}{2} \int_0^t f(\pi^{t,x,a,u}_s,\gamma^{t,a,u}_s,u_s)\,\rd s - C + g(\pi^{t,x,a,u}_0,\gamma^{t,a,u}_0)\bigg\}
\end{equation}
and, since $f$ and $g$ are bounded below, the same is true of $\kappa$.

\textbf{Step 2.}
Now let $\Delta$ be a compact subset of the domain $\cD$. By the definition of $\cD$, for each point $(\bar{t},\bar{x},\bar{a}) \in \Delta$ there exists a control $\bar{u} \in \cU$ such that $\pi^{\bar{t},\bar{x},\bar{a},\bar{u}}_s \in S^m$ for all $s \in [0,\bar{t}]$. Moreover, by continuity, there exists a compact subset $\Xi^{\bar{t},\bar{x},\bar{a}}$ of $S^m$ and an open neighbourhood $\cO^{\bar{t},\bar{x},\bar{a}}$ of $(\bar{t},\bar{x},\bar{a})$ such that $\pi^{t,x,a,\bar{u}}_s \in \Xi^{\bar{t},\bar{x},\bar{a}}$ for all $(t,x,a) \in \cO^{\bar{t},\bar{x},\bar{a}}$ and all $s \in [0,t]$.

Since $\{\cO^{\bar{t},\bar{x},\bar{a}} : (\bar{t},\bar{x},\bar{a}) \in \Delta\}$ is an open cover for the compact set $\Delta$, there exists a finite collection of points $(\bar{t},\bar{x},\bar{a}) \in \Delta$ and corresponding controls $\bar{u}$ such that $\Delta \subset \cup_{(\bar{t},\bar{x},\bar{a})} \cO^{\bar{t},\bar{x},\bar{a}}$. The finite union $\Xi := \cup_{(\bar{t},\bar{x},\bar{a})} \Xi^{\bar{t},\bar{x},\bar{a}}$ is clearly compact. Moreover, we have shown that: for any $(t,x,a) \in \Delta$ there exists a control $\bu \in \cU$ from our finite collection such that $\pi^{t,x,a,\bu}_s \in \Xi$ for all $s \in [0,t]$.

Since each control $\bar{u} \colon [0,T] \to \R^k$ is bounded, the finite collection of controls specified above is uniformly bounded. Thus, there exists a compact set $K \subset S^m \times \R^k$ such that, for any $(t,x,a) \in \Delta$, there exists a control $\bu \in \cU$ such that $(\pi^{t,x,a,\bu}_s,\gamma^{t,a,\bu}_s) \in K$ for all $s \in [0,t]$.

Since $f$ and $g$ are assumed to be continuous, they are locally bounded, and hence bounded on $K$. Thus, using \eqref{eq:HMMpsiintlocbound} again, we have
\begin{equation*}
\sup_{(t,x,a) \in \Delta} \kappa(t,x,a) \leq \sup_{(t,x,a) \in \Delta} \bigg(\frac{3}{2}\int_0^t f(\pi^{t,x,a,\bu}_s,\gamma^{t,a,\bu}_s,\bu_s)\,\rd s + C + g(\pi^{t,x,a,\bu}_0,\gamma^{t,a,\bu}_0)\bigg) < \infty,
\end{equation*}
so that $\kappa$ is bounded above on $\Delta$, and hence locally bounded above on $\cD$.
\end{proof}

\begin{corollary}\label{corollary HMM restrict controls}
Let $\Delta$ be a compact subset of $\cD$. Under Assumptions~\ref{assumption uncertainty parameterisation} and \ref{assumptions f g}, there exists a bound $M > 0$ and a compact set $K \subset S^m \times \R^k$ such that, when taking the infimum in \eqref{eq:valuefuncHMM} for $(t,x,a) \in \Delta$, one may restrict to controls $u \in \cU_{M,K}$ without changing the value of $\kappa(t,x,a)$, where\footnote{We suppress the dependency of $\cU_{M,K}$ on the point $(t,x,a)$ in our notation.}
\begin{equation}\label{eq:controlMbound}
\cU_{M,K} := \big\{u \in \cU \, : \, \|\gamma^{t,a,u}\|_{\ptvarzt} \leq M, \textnormal{ and } (\pi^{t,x,a,u}_s,\gamma^{t,a,u}_s) \in K \textnormal{ for all } s \in [0,t]\big\}.
\end{equation}
\end{corollary}

\begin{proof}
We recall the inequality \eqref{eq:kappaboundedbelow}, which reads:
\begin{equation}\label{eq:kappaboundedbelow again}
\kappa(t,x,a) \geq \inf_{u \in \cU} \bigg\{\frac{1}{2} \int_0^t f(\pi^{t,x,a,u}_s,\gamma^{t,a,u}_s,u_s)\,\rd s - C + g(\pi^{t,x,a,u}_0,\gamma^{t,a,u}_0)\bigg\}
\end{equation}
for some constant $C$. Since $\kappa$ is bounded above on $\Delta$, and since $g$ is bounded below, we infer an upper bound on $\int_0^t f(\pi^{t,x,a,u}_s,\gamma^{t,a,u}_s,u_s)\,\rd s$. Recalling \eqref{eq:bound gamma int u} and \eqref{eq:Afepsilonbound}, we see that this then translates into an upper bound $M$, say, on $\|\gamma^{t,a,u}\|_{\ptvarzt}$.

Since $\kappa$ is bounded above on $\Delta$, and since $f$ is bounded below, we similarly infer from \eqref{eq:kappaboundedbelow again} an upper bound on $g(\pi^{t,x,a,u}_0,\gamma^{t,a,u}_0)$. Since $g \in  C^\uparrow(S^m \times \R^k;\R)$, this implies the existence of a compact set $\Xi \subset S^m \times \R^k$, such that, for terminal values $(t,x,a) \in \Delta$, we may restrict to controls $u$ such that $(\pi^{t,x,a,u}_0,\gamma^{t,a,u}_0) \in \Xi$.

Since both the initial and terminal values of the state variables $\pi^{t,x,a,u}, \gamma^{t,a,u}$ are then restricted to the compact sets $\Xi$ and $\Delta$ respectively, and since we know that we may restrict to controls $u$ such that $\|\gamma^{t,a,u}\|_{\ptvarzt} \leq M$, we conclude that the entire path $s \mapsto (\pi^{t,x,a,u}_s,\gamma^{t,a,u}_s)$ may be restricted to a compact set $K$.
\end{proof}

\begin{notation}
Recall Notation~\ref{notation lesssim}. Henceforth, whenever we have identified a compact subset $\Delta \subset \cD$, we shall allow the multiplicative constant indicated by the symbol $\lesssim$ to also depend on the constant $M$ and on the set $K$ in \eqref{eq:controlMbound}.
\end{notation}

\subsection{Regularity of the value function}

We have the following dynamic programming principle.

\begin{lemma}\label{lemma HMM DPP}
For any $(t,x,a) \in \cD$ and $r \in [0,t]$, the value function $\kappa$ (as defined in \eqref{eq:valuefuncHMM} above), satisfies
\begin{align}
&\kappa(t,x,a)\label{eq:HMM DPP}\\
&= \inf_{u \in \cU} \bigg\{\int_r^t f(\pi^{t,x,a,u}_s,\gamma^{t,a,u}_s,u_s)\,\rd s + \int_r^t \psi(\pi^{t,x,a,u}_s,\gamma^{t,a,u}_s)\,\rd \bY_s + \kappa(r,\pi^{t,x,a,u}_r,\gamma^{t,a,u}_r)\bigg\}.\nonumber
\end{align}
\end{lemma}

The result of Lemma~\ref{lemma HMM DPP} follows the same proof as that of Theorem~2.1 in \cite[Chapter~4]{YongZhou1999}. In particular, the rough integrals appearing in the controlled dynamics and value function do not cause any additional difficulty.

\begin{proposition}\label{prop kappa loc Lip x}
Under Assumptions~\ref{assumption uncertainty parameterisation} and \ref{assumptions f g}, the value function $\kappa = \kappa(t,x,a)$ is locally Lipschitz continuous in $(x,a)$, uniformly in $t \in [0,T]$.
\end{proposition}

\begin{proof}
Let $\Delta$ be a compact subset of $\cD$, and let $(t,x,a), (t,\tx,\ta) \in \Delta$. Let $u \in \cU$. By Corollary~\ref{corollary HMM restrict controls}, we may assume that $u \in \cU_{M,K}$, i.e.~$\|\gamma^{t,a,u}\|_{\ptvarzt} \leq M$ for some $M > 0$, and there exists a compact subset $K \subset S^m \times \R^k$, such that $(\pi^{t,x,a,u}_s,\gamma^{t,a,u}_s) \in K$ for all $(t,x,a) \in \Delta$ and $s \in [0,t]$. By Corollary~\ref{corollary reverse time}, we have
\begin{align}
\|\pi^{t,x,a,u} - \pi^{t,\tx,\ta,u}\|_{\pvarzt} &\lesssim |x - \tx| + |a - \ta|,\label{eq:picontboundxtx}\\
\bigg\|\int_0^\cdot \psi(\pi^{t,x,a,u}_s,\gamma^{t,a,u}_s)\,\rd \bY_s - \int_0^\cdot \psi(\pi^{t,\tx,\ta,u}_s,\gamma^{t,\ta,u}_s)\,\rd \bY_s\bigg\|_{\pvarzt} &\lesssim |x - \tx| + |a - \ta|.\label{eq:intpsicontboundxtx}
\end{align}

Since we have restricted to the compact set $K$, we may then take the functions $f$ and $g$ to be Lipschitz in $(x,a)$. Using \eqref{eq:picontboundxtx} and \eqref{eq:intpsicontboundxtx}, we then have
\begin{align*}
\big|\kappa&(t,x,a) - \kappa(t,\tx,\ta)\big|\\
&= \sup_{u \in \cU_{M,K}} \bigg|\int_0^t \big(f(\pi^{t,x,a,u}_s,\gamma^{t,a,u}_s,u_s) - f(\pi^{t,\tx,\ta,u}_s,\gamma^{t,\ta,u}_s,u_s)\big)\hspace{1pt}\rd s\\
&\hspace{50pt} + \int_0^t \psi(\pi^{t,x,a,u}_s,\gamma^{t,a,u}_s)\,\rd \bY_s - \int_0^t \psi(\pi^{t,\tx,\ta,u}_s,\gamma^{t,\ta,u}_s)\,\rd \bY_s\\
&\hspace{70pt} + g(\pi^{t,x,a,u}_0,\gamma^{t,a,u}_0) - g(\pi^{t,\tx,\ta,u}_0,\gamma^{t,\ta,u}_0)\bigg|\\
&\lesssim \sup_{u \in \cU_{M,K}} \bigg(\int_0^t |\pi^{t,x,a,u}_s - \pi^{t,\tx,\ta,u}_s| + |\gamma^{t,a,u}_s - \gamma^{t,\ta,u}_s|\,\rd s\\
&\hspace{50pt} + |x - \tx| + |a - \ta| + |\pi^{t,x,a,u}_0 - \pi^{t,\tx,\ta,u}_0| + |\gamma^{t,a,u}_0 - \gamma^{t,\ta,u}_0|\bigg)\\
&\lesssim |x - \tx| + |a - \ta|,
\end{align*}
and the result follows.
\end{proof}

\begin{proposition}\label{prop kappa cts t}
Under Assumptions~\ref{assumption uncertainty parameterisation} and \ref{assumptions f g}, the value function $\kappa = \kappa(t,x,a)$ is continuous in $t$, with a local modulus of continuity which is uniform in $(x,a)$.
\end{proposition}

\begin{proof}
Let $\Delta$ be a compact and convex subset of $\cD$. Let $(r,x,a), (t,x,a) \in \Delta$ with $r \leq t$. Note that then $(s,x,a) \in \Delta$ for all $s \in [r,t]$ by convexity. By Corollary~\ref{corollary HMM restrict controls}, we may restrict to controls $u \in \cU_{M,K}$, so that $\|\gamma^{t,a,u}\|_{\ptvarzt} \leq M$ for some $M > 0$, and there exists a compact subset $K \subset S^m \times \R^k$, such that $(\pi^{t,x,a,u}_s,\gamma^{t,a,u}_s) \in K$ for all $(t,x,a) \in \Delta$ and $s \in [0,t]$. Similarly to the proof of Theorem~2.2 in Bardi and Da Lio \cite{BardiDaLio1997}, by Lemma~\ref{lemma HMM DPP} we can further restrict to controls $u \in \cU_{M,K}$ such that
\begin{align}
&\int_r^t f(\pi^{t,x,a,u}_s,\gamma^{t,a,u}_s,u_s)\,\rd s\nonumber\\
&\leq \bigg|\int_r^t f(\pi^{t,x,a,0}_s,\gamma^{t,a,0}_s,0)\,\rd s + \int_r^t \psi(\pi^{t,x,a,0}_s,\gamma^{t,a,0}_s)\,\rd \bY_s - \int_r^t \psi(\pi^{t,x,a,u}_s,\gamma^{t,a,u}_s)\,\rd\bY_s\bigg|\nonumber\\
&\quad + \big|\kappa(r,\pi^{t,x,a,0}_r,\gamma^{t,a,0}_r) - \kappa(r,\pi^{t,x,a,u}_r,\gamma^{t,a,u}_r)\big|,\label{eq:bound f with DPP}
\end{align}
where $0 \in \cU$ is the zero control. We aim to bound each of the terms on the right-hand side. By part (ii) of Theorem~\ref{theoremHMMRDE}, we have
\begin{equation}\label{eq:bound int psi rt}
\bigg|\int_r^t \psi(\pi^{t,x,a,u}_s,\gamma^{t,a,u}_s)\,\rd\bY_s\bigg| \leq \bigg\|\int_0^\cdot \psi(\pi^{t,x,a,u}_s,\gamma^{t,a,u}_s)\,\rd\bY_s\bigg\|_{p,[r,t]} \lesssim \ver{\bY}_{p,[r,t]},
\end{equation}
and similarly with $u$ replaced by $0$.

Since the path $s \mapsto (\pi^{t,x,a,0}_s,\gamma^{t,a,0}_s)$ then lives in a compact set, and since $f$ is locally bounded, it follows that
\begin{equation}\label{eq:bound f 0}
\bigg|\int_r^t f(\pi^{t,x,a,0}_s,\gamma^{t,a,0}_s,0)\,\rd s\bigg| \lesssim |t - r|.
\end{equation}
By Proposition~\ref{prop kappa loc Lip x}, since we have restricted to a compact set, we may take $\kappa$ to be Lipschitz in $(x,a)$. We then have that
\begin{equation*}
\big|\kappa(r,\pi^{t,x,a,0}_r,\gamma^{t,a,0}_r) - \kappa(r,\pi^{t,x,a,u}_r,\gamma^{t,a,u}_r)\big| \lesssim |\pi^{t,x,a,0}_r - \pi^{t,x,a,u}_r| + |\gamma^{t,a,0}_r - \gamma^{t,a,u}_r|.
\end{equation*}
By Corollary~\ref{corollary reverse time}, we have
$$|\pi^{t,x,a,0}_r - \pi^{t,x,a,u}_r| \leq \|\pi^{t,x,a,0} - \pi^{t,x,a,u}\|_{\pvarrt} \lesssim \|\gamma^{t,a,0} - \gamma^{t,a,u}\|_{\ptvarrt},$$
and hence
\begin{equation}\label{eq:bound kappa lip r}
\big|\kappa(r,\pi^{t,x,a,0}_r,\gamma^{t,a,0}_r) - \kappa(r,\pi^{t,x,a,u}_r,\gamma^{t,a,u}_r)\big| \lesssim \|\gamma^{t,a,0} - \gamma^{t,a,u}\|_{\onevarrt} = \int_r^t |u_s|\,\rd s.
\end{equation}
By part (ii) of Assumption~\ref{assumptions f g}, for any $\epsilon > 0$, there exists a constant $C_\epsilon$ such that
\begin{equation}\label{eq:Afepsilonbound 2}
|u| \leq C_\epsilon + \epsilon f(x,a,u)
\end{equation}
for all $(x,a,u) \in S^m \times \R^k \times \R^k$. Combining \eqref{eq:bound int psi rt}--\eqref{eq:Afepsilonbound 2} with \eqref{eq:bound f with DPP}, we obtain
\begin{equation*}
\int_r^t f(\pi^{t,x,a,u}_s,\gamma^{t,a,u}_s,u_s)\,\rd s \lesssim (1 + C_\epsilon)|t - r| + \ver{\bY}_{p,[r,t]} + \epsilon\int_r^t f(\pi^{t,x,a,u}_s,\gamma^{t,a,u}_s,u_s)\,\rd s,
\end{equation*}
and choosing $\epsilon$ sufficiently small, we deduce that
\begin{equation}\label{eq:intflessincrement}
\int_r^t f(\pi^{t,x,a,u}_s,\gamma^{t,a,u}_s,u_s)\,\rd s \lesssim |t - r| + \ver{\bY}_{p,[r,t]}.
\end{equation}
It also follows from the above that
\begin{align}
|\pi^{t,x,a,u}_r - x| + |\gamma^{t,a,u}_r - a| &\leq \bigg|\int_r^t b(\pi^{t,x,a,u}_s,\gamma^{t,a,u}_s)\,\rd s + \int_r^t \phi(\pi^{t,x,a,u}_s,\gamma^{t,a,u}_s)\,\rd \bY_s\bigg|\nonumber\\
&\quad + \|\gamma^{t,a,u} - \gamma^{t,a,0}\|_{\onevarrt}\nonumber\\
&\lesssim |t - r| + \ver{\bY}_{\pvarrt},\label{eq:bound pi x gamma a}
\end{align}
where we used the fact that $b$ is uniformly bounded, and that \eqref{eq:bound int psi rt} also holds with $\psi$ replaced by $\phi$. By Lemma~\ref{lemma HMM DPP}, we can take a sequence of controls $(u^n)_{n \geq 1} \subset \cU_{M,K}$ such that
\begin{align*}
\kappa(t,x,a) = \lim_{n \to \infty} \bigg(&\int_r^t f(\pi^{t,x,a,u^n}_s,\gamma^{t,a,u^n}_s,u^n_s)\,\rd s\\
&+ \int_r^t \psi(\pi^{t,x,a,u^n}_s,\gamma^{t,a,u^n}_s)\,\rd \bY_s + \kappa(r,\pi^{t,x,a,u^n}_r,\gamma^{t,a,u^n}_r)\bigg).
\end{align*}
Using \eqref{eq:bound int psi rt} and \eqref{eq:intflessincrement}, and the fact that $\kappa$ is locally Lipschitz in $(x,a)$, we have
\begin{align*}
&\bigg|\int_r^t f(\pi^{t,x,a,u^n}_s,\gamma^{t,a,u^n}_s,u^n_s)\,\rd s + \int_r^t \psi(\pi^{t,x,a,u^n}_s,\gamma^{t,a,u^n}_s)\,\rd \bY_s\\
&\quad + \kappa(r,\pi^{t,x,a,u^n}_r,\gamma^{t,a,u^n}_r) - \kappa(r,x,a)\bigg|\\
&\lesssim |t - r| + \ver{\bY}_{p,[r,t]} + |\pi^{t,x,a,u^n}_r - x| + |\gamma^{t,a,u^n}_r - a|.
\end{align*}
Using \eqref{eq:bound pi x gamma a}, and taking the limit as $n \to \infty$, we obtain
\begin{equation}\label{eq:kappa cts in t}
\big|\kappa(t,x,a) - \kappa(r,x,a)\big| \lesssim |t - r| + \ver{\bY}_{p,[r,t]},
\end{equation}
which implies that $\kappa$ is continuous in $t$, uniformly in $x$ and $a$.
\end{proof}

\begin{notation}\label{notation C up cD R}
We write
$$d(x,\pa Q_t) := \inf_{y \in \pa Q_t} |x - y|$$
for the distance of a point $x \in Q_t$ to the boundary $\pa Q_t := \overline{Q}_t \setminus Q_t$ of $Q_t$.

We denote by $C^\uparrow(\cD;\R)$ the space of continuous functions $v \colon \cD \to \R$ which explode near the boundary of $Q_t$ and for large values of $a \in \R^k$, that is, such that
\begin{align}
\inf_{t \in [0,T],\, a \in \R^k} \inf \big\{v(t,x,a)\, \big|\ x \in Q_t,\, d(x,\pa Q_t) < \delta\big\} \, \longrightarrow \, \infty \qquad &\text{as} \quad \delta \, \longrightarrow \, 0^+,\label{eq:notation v infty x}\\
\inf_{t \in [0,T],\, x \in Q_t} v(t,x,a) \, \longrightarrow \, \infty \qquad &\text{as} \quad |a| \, \longrightarrow \, \infty.\label{eq:notation v infty a}
\end{align}
\end{notation}

\begin{proposition}\label{prop kappa blows up}
Under Assumptions~\ref{assumption uncertainty parameterisation} and \ref{assumptions f g}, we have that $\kappa \in C^\uparrow(\cD;\R)$.
\end{proposition}

\begin{proof}
We first note that combining Propositions~\ref{prop kappa loc Lip x} and \ref{prop kappa cts t} yields joint continuity of $\kappa$ in all its variables. It remains to establish the conditions \eqref{eq:notation v infty x} and \eqref{eq:notation v infty a}.

\textbf{Step 1.}
We recall the inequality \eqref{eq:kappaboundedbelow}, which reads:
\begin{equation}\label{eq:kappaboundedbelow 2}
\kappa(t,x,a) \geq \inf_{u \in \cU} \bigg\{\frac{1}{2} \int_0^t f(\pi^{t,x,a,u}_s,\gamma^{t,a,u}_s,u_s)\,\rd s - C_1 + g(\pi^{t,x,a,u}_0,\gamma^{t,a,u}_0)\bigg\}
\end{equation}
for some constant $C_1$. By part (ii) of Assumption~\ref{assumptions f g}, there exists another constant $C$ such that
\begin{equation}\label{eq:u leq f}
|u| \leq C + f(x,a,u)
\end{equation}
for all $(x,a,u) \in S^m \times \R^k \times \R^k$. Recalling that $\rd \gamma^{t,a,u}_s = u_s\,\rd s$ and $\gamma^{t,a,u}_t = a$, we obtain
$$|a| \leq |\gamma^{t,a,u}_0| + \int_0^t |u_s|\,\rd s \lesssim |\gamma^{t,a,u}_0| + 1 + \int_0^t f(\pi^{t,x,a,u}_s,\gamma^{t,a,u}_s,u_s)\,\rd s.$$
Thus, as $|a| \to \infty$, it must be the case that either $\int_0^t f(\pi^{t,x,a,u}_s,\gamma^{t,a,u}_s,u_s)\,\rd s \to \infty$, or $|\gamma^{t,a,u}_0| \to \infty$, and in the latter case it then follows from part (iii) of Assumption~\ref{assumptions f g} that $g(\pi^{t,x,a,u}_0,\gamma^{t,a,u}_0) \to \infty$. Since $f$ and $g$ are both bounded below, it follows from \eqref{eq:kappaboundedbelow 2} that $\inf_{(t,x)} \kappa(t,x,a) \to \infty$ as $|a| \to \infty$, i.e.~that \eqref{eq:notation v infty a} holds for $v = \kappa$.

\textbf{Step 2.}
Let us now assume for a contradiction that \eqref{eq:notation v infty x} does not hold for $v = \kappa$. We then infer the existence of a sequence $((t^n,x^n,a^n))_{n \geq 1} \subset \cD$ and a constant $C_2$, such that $d(x^n,\pa Q_{t^n}) \to 0$ as $n \to \infty$, and $\kappa(t^n,x^n,a^n) \leq C_2$ for all $n \geq 1$.

By \eqref{eq:kappaboundedbelow 2}, for each $n \geq 1$ there exists a control $u^n \in \cU$ such that
\begin{align}
&\frac{1}{2}\int_0^{t^n} f(\pi^{t^n,x^n,a^n,u^n}_s,\gamma^{t^n,a^n,u^n}_s,u^n_s)\,\rd s - C_1 + g(\pi^{t^n,x^n,a^n,u^n}_0,\gamma^{t^n,a^n,u^n}_0)\nonumber\\
&< \kappa(t^n,x^n,a^n) + 1 \leq C_2 + 1.\label{eq:seq f g bounded}
\end{align}
We have two possibilities. Namely, either there exists a subsequence $(n_j)_{j \geq 1}$ such that
\begin{equation}\label{eq:int u n j to infty}
\int_0^{t^{n_j}} |u^{n_j}_s|\,\rd s\, \longrightarrow \, \infty \qquad \text{as} \quad j \, \longrightarrow \, \infty,
\end{equation}
or there does not. If there does exist such a subsequence, then it follows from \eqref{eq:u leq f} that
$$\int_0^{t^{n_j}} f(\pi^{t^{n_j},x^{n_j},a^{n_j},u^{n_j}}_s,\gamma^{t^{n_j},a^{n_j},u^{n_j}}_s,u^{n_j}_s)\,\rd s \, \longrightarrow \, \infty \qquad \text{as} \quad j \, \longrightarrow \, \infty,$$
contradicting \eqref{eq:seq f g bounded} (since $g$ is bounded below).

\textbf{Step 3.}
If there does not exist a subsequence such that \eqref{eq:int u n j to infty} holds, then it follows immediately that there exists a constant $M > 0$ such that
$$\|\gamma^{t^n,a^n,u^n}\|_{\frac{p}{2},[0,t^n]} \leq \|\gamma^{t^n,a^n,u^n}\|_{1,[0,t^n]} = \int_0^{t^n} |u^n_s|\,\rd s \leq M \qquad \text{for every} \quad n \geq 1.$$
We can then apply Corollary~\ref{corollary reverse time} to deduce the existence of a single constant $C$ such that
$$\|\pi^{t^n,x^n,a^n,u^n} - \pi^{t^n,z^n,a^n,u^n}\|_{p,[0,t^n]} \leq C|x^n - z^n|,$$
so that
$$|\pi^{t^n,x^n,a^n,u^n}_0 - \pi^{t^n,z^n,a^n,u^n}_0| \leq (C + 1)|x^n - z^n|$$
for any points $z^n \in Q_{t^n}$.

Since terminal values of $\pi$ on (or outside) the boundary $\pa Q_{t^n}$ result in initial values outside the domain $S^m$ (by the definition of $Q_{t^n}$), we may choose a terminal value $z^n$ close to $x^n$, but also close enough to the boundary $\pa Q_{t^n}$ to ensure that the initial value $\pi^{t^n,z^n,a^n,u^n}_0$ is arbitrarily close to the boundary $\pa S^m$. More precisely, we choose the points $(z^n)_{n \geq 1}$ such that $|x^n - z^n| \leq d(x^n,\pa Q_{t^n})$, and such that the corresponding initial values satisfy $d(\pi^{t^n,z^n,a^n,u^n}_0,\pa S^m) \to 0$ as $n \to \infty$. In particular, we then have that
$$|\pi^{t^n,x^n,a^n,u^n}_0 - \pi^{t^n,z^n,a^n,u^n}_0| \leq (C + 1)|x^n - z^n| \leq (C + 1)d(x^n,\pa Q_{t^n}) \, \longrightarrow \, 0 \quad \text{as} \quad n\, \longrightarrow \, \infty.$$
Since $d(\pi^{t^n,z^n,a^n,u^n}_0,\pa S^m) \to 0$ as $n \to \infty$, we can find a sequence $(y^n)_{n \geq 1} \subset \pa S^m$ such that $|\pi^{t^n,z^n,a^n,u^n}_0 - y^n| \to 0$ as $n \to \infty$. Then
\begin{align*}
d(\pi^{t^n,x^n,a^n,u^n}_0,\pa S^m) &\leq |\pi^{t^n,x^n,a^n,u^n}_0 - y^n|\\
&\leq |\pi^{t^n,x^n,a^n,u^n}_0 - \pi^{t^n,z^n,a^n,u^n}_0| + |\pi^{t^n,z^n,a^n,u^n}_0 - y^n| \, \longrightarrow \, 0
\end{align*}
as $n \to \infty$, and hence, since $g \in C^\uparrow(S^m \times \R^k;\R)$, we deduce that
$$g(\pi^{t^n,x^n,a^n,u^n}_0,\gamma^{t^n,a^n,u^n}_0) \, \longrightarrow \, \infty \qquad \text{as} \quad n \, \longrightarrow \, \infty,$$
contradicting \eqref{eq:seq f g bounded} (since $f$ is bounded below).
\end{proof}

\section{Hamilton--Jacobi equations}\label{sec HJ equations}

\subsection{A smooth regularization}

Our main aim is to establish the function $\kappa$ (recall \eqref{eq:defnkappaHMM} or \eqref{eq:valuefuncHMM}) as the solution of a rough HJ equation (namely \eqref{eq:HMMroughHJB} below). As in Diehl et al.~\cite{DiehlFrizGassiat2017}, we first approximate the rough path $\bY$ by a smooth path $\eta$. Having solved the associated classical control problem, as is a standard strategy for rough ODEs and PDEs, we can define solutions to our HJ equation for genuinely rough driving paths by taking the closure of smooth paths in rough path topology.

Accordingly, given a smooth path $\eta \colon [0,T] \to \R^d$, we define the approximate value function:
\begin{align}
&\kappa^\eta(t,x,a)\label{eq:valuefuncsmoothHMM}\\
&= \inf_{u \in \cU} \bigg\{\int_0^t f(\pi^{t,x,a,u}_s,\gamma^{t,a,u}_s,u_s)\,\rd s + \int_0^t \psi(\pi^{t,x,a,u}_s,\gamma^{t,a,u}_s)\,\rd \eta_s + g(\pi^{t,x,a,u}_0,\gamma^{t,a,u}_0)\bigg\},\nonumber
\end{align}
where here the state variables $\pi, \gamma$ satisfy
\begin{align}
\rd \pi^{t,x,a,u}_s &= b(\pi^{t,x,a,u}_s,\gamma^{t,a,u}_s)\hspace{1pt}\rd s + \phi(\pi^{t,x,a,u}_s,\gamma^{t,a,u}_s) \, \rd \eta_s, & \pi^{t,x,a,u}_t &= x,\label{eq:contrdynamicssmoothHMM}\\
\rd \gamma^{t,a,u}_s &= u_s\,\rd s, & \gamma^{t,a,u}_t &= a.\nonumber
\end{align}

In the following, we will keep the assumption that $\ver{\bme}_{p,[0,T]} \leq L$, where $\bme$ is defined as the canonical lift of $\eta$, i.e.~$\bme = (\eta,\eta^{(2)}) \in \grp$ with
\begin{equation}\label{eq:HMM eta2 defn}
\eta^{(2)}_{s,t} := \int_s^t \eta_{s,r}\otimes\rd\eta_r \qquad \text{for all} \quad (s,t) \in \simplex,
\end{equation}
where, since $\eta$ is smooth, the integral may be understood in the Riemann--Stieltjes sense.

Recall Notations~\ref{notation Qt cD} and \ref{notation C up cD R}. Since in general the set of plausible posteriors $Q_t$, and hence also the domain $\cD$ on which the value function $\kappa$ is finite, may depend on the observation path, we will correspondingly write $Q^\eta_t := \{x \in S^m \, | \, \kappa^\eta(t,x,a) < \infty\} = \big\{x \in S^m \ \big| \ \exists u \in \cU \hspace{7pt} \text{such that} \hspace{7pt} \pi^{t,x,a,u}_0 \in S^m\big\}$ and $\cD^\eta := \cup_{t \in [0,T]} (\{t\} \times Q^\eta_t \times \R^k)$. Note however that, by the stability of solutions to rough differential equations (Corollary~\ref{corollary reverse time}), we have that $\cD^\eta \to \cD$ as $\bme \to \bY$ in the obvious sense; in particular, if $(t,x,a) \in \cD$, then $(t,x,a) \in \cD^\eta$ whenever $\|\bme;\bY\|_p$ is sufficiently small.

By simply replacing $\bY$ by $\bme$ in the corresponding proofs, the approximate value function $\kappa^\eta$ inherits all the properties established in the previous section, namely:
\begin{itemize}
\item $\kappa^\eta$ is bounded below, and locally bounded above on $\cD^\eta$,
\item $\kappa^\eta$ satisfies the dynamic programming principle, i.e.~\eqref{eq:HMM DPP} with $\bY$ replaced by $\eta$,
\item $\kappa^\eta = \kappa^\eta(t,x,a)$ is locally Lipschitz continuous in $(x,a)$, uniformly in $t \in [0,T]$,
\item $\kappa^\eta = \kappa^\eta(t,x,a)$ is continuous in $t$, with a local modulus of continuity which is uniform in $(x,a)$,
\item $\kappa^\eta \in C^\uparrow(\cD^\eta;\R)$.
\end{itemize}
Moreover, $\kappa^\eta$ is actually Lipschitz continuous in $t$, locally uniformly on $\cD^\eta$. To see this, we recall the estimate \eqref{eq:kappa cts in t} from the proof of Proposition~\ref{prop kappa cts t}, which, replacing $\bY$ by $\bme$, reads:
\begin{equation}\label{eq:kappa cts in t 2}
\big|\kappa^\eta(t,x,a) - \kappa^\eta(r,x,a)\big| \lesssim |t - r| + \ver{\bme}_{p,[r,t]}
\end{equation}
for all $(r,x,a), (t,x,a)$ in a given compact and convex subset of $\cD^\eta$. Since $\eta$ is smooth, a short calculation shows that $\ver{\bme}_{p,[r,t]} = \|\eta\|_{p,[r,t]} + \|\eta^{(2)}\|_{\frac{p}{2},[r,t]} \lesssim (1 + \|\eta\|_\infty)\|\dot{\eta}\|_\infty|t - r|$. Substituting this into \eqref{eq:kappa cts in t 2}, we deduce that $\kappa^\eta$ is Lipschitz in $t$.

\subsection{A smooth HJ equation}

We will return to the (rough) value function $\kappa$ in Section~\ref{sec rough HJ eqn} below. For now we will restrict our attention to the smoothed version $\kappa^\eta$, as defined in \eqref{eq:valuefuncsmoothHMM}, and introduce the associated HJ equation:
\begin{align}
\frac{\pa\kappa^\eta}{\pa t} + b\cdot\nabla_x\kappa^\eta + \sup_{u \in \R^k} \big\{u\cdot\nabla_a\kappa^\eta - f\big\} + \big(\phi\cdot\nabla_x\kappa^\eta - \psi\big)\deta &= 0 & &\text{on}\hspace{10pt} \cD^\eta,\label{eq:HMMsmoothHJB}\\
\kappa^\eta(0,\cdot\hspace{1pt},\cdot) &= g & &\text{on}\hspace{10pt} S^m \times \R^k,\label{eq:HMMsmoothHJBinitcond}
\end{align}
for a smooth approximation $\eta$ of $\bY$, where as usual $\deta$ denotes the derivative of $\eta$.

\begin{remark}
Since the spatial variable $x$ is confined to the simplex $S^m \subset \R^m$, it may not seem meaningful to consider taking the gradient $\nabla_x$. However, since the coefficients $b$ and $\phi$ always remain directed within the simplex, the directional derivatives $b\cdot\nabla_x$ and $\phi\cdot\nabla_x$ always exist.
\end{remark}

In the following we consider solutions of \eqref{eq:HMMsmoothHJB}--\eqref{eq:HMMsmoothHJBinitcond} in the sense of viscosity solutions. The unfamiliar reader is referred to Barles \cite{AchdouBarlesLitvinovIshii2013} or Crandall, Ishii and Lions \cite{CrandallIshiiLions1992} for a detailed explanation.

\begin{definition}\label{defn viscosity soln}
We say that a continuous function $v \colon \cD^\eta \to \R$ is a viscosity subsolution (resp.~supersolution) of \eqref{eq:HMMsmoothHJB}--\eqref{eq:HMMsmoothHJBinitcond} if $v(0,x,a) \leq(\text{resp.} \geq)\ g(x,a)$ for all $(x,a) \in S^m \times \R^k$, and, for any point $(t,x,a) \in \cD^\eta$ with $t \in (0,T]$,
\begin{align}
\frac{\pa \vp}{\pa t}(t,x,a) + b(x,a)\cdot\nabla_x \vp(&t,x,a) + \sup_{u \in \R^k} \big\{u\cdot\nabla_a \vp(t,x,a) - f(x,a,u)\big\}\nonumber\\
&+ \big(\phi(x,a)\cdot\nabla_x \vp(t,x,a) - \psi(x,a)\big)\deta_t \leq\hspace{-1pt}(\text{resp.} \geq)\hspace{2pt} 0\label{eq:viscosity soln}
\end{align}
for every smooth function $\vp \colon \cD^\eta \to \R$ such that $v - \vp$ has a local maximum (resp.~local minimum) at the point $(t,x,a)$. We say that $v$ is a viscosity solution if it is both a viscosity subsolution and a viscosity supersolution.
\end{definition}

\begin{proposition}\label{prop value func is vis soln}
Under Assumptions~\ref{assumption uncertainty parameterisation} and \ref{assumptions f g}, the approximate value function $\kappa^\eta$ is a viscosity solution of \eqref{eq:HMMsmoothHJB}--\eqref{eq:HMMsmoothHJBinitcond}.
\end{proposition}

The proof of Proposition~\ref{prop value func is vis soln} is standard---see Theorem~2.5 in \cite[Chapter~4]{YongZhou1999} or Proposition~4.9 in \cite{AllanCohen2019} for details in analogous settings.

The result of Proposition~\ref{prop kappa blows up} above shows that the (approximate) value function explodes near the boundary of the domain $\cD^\eta$. In fact, this `explosive boundary condition' is precisely the extra condition needed to obtain uniqueness for the corresponding HJ equation.

\begin{theorem}\label{theorem HMM kappaeta unique soln}
Under Assumptions~\ref{assumption uncertainty parameterisation} and \ref{assumptions f g}, the approximate value function $\kappa^\eta$ is both the minimal viscosity supersolution and the maximal viscosity subsolution of \eqref{eq:HMMsmoothHJB}--\eqref{eq:HMMsmoothHJBinitcond} in the class $C^\uparrow(\cD^\eta;\R)$, and is thus the unique viscosity solution of \eqref{eq:HMMsmoothHJB}--\eqref{eq:HMMsmoothHJBinitcond} in the class $C^\uparrow(\cD^\eta;\R)$.
\end{theorem}

\begin{proof}
We will prove the minimality of $\kappa^\eta$ among viscosity supersolutions of \eqref{eq:HMMsmoothHJB}--\eqref{eq:HMMsmoothHJBinitcond}. The proof of maximality among subsolutions follows a similar argument, and is hence omitted for brevity.

Let $v \in C^\uparrow(\cD^\eta;\R)$ be another viscosity supersolution of \eqref{eq:HMMsmoothHJB}--\eqref{eq:HMMsmoothHJBinitcond}, and let $\Delta$ be a compact subset of $\cD^\eta$. Recalling \eqref{eq:HMMpsiintlocbound} and the fact that $f$ is bounded below, we have
\begin{align*}
&\int_0^t f(\pi^{t,x,a,u}_s,\gamma^{t,a,u}_s,u_s)\,\rd s + \int_0^t \psi(\pi^{t,x,a,u}_s,\gamma^{t,a,u}_s)\,\rd \eta_s + g(\pi^{t,x,a,u}_0,\gamma^{t,a,u}_0)\\
&\geq g(\pi^{t,x,a,u}_0,\gamma^{t,a,u}_0) - C
\end{align*}
for some constant $C$. Since $\kappa^\eta$ is locally bounded above (by Lemma~\ref{lemma value func loc bdd}), and hence bounded above on $\Delta$, we infer that there exists a bound $\lambda > 0$ such that, for $(t,x,a) \in \Delta$, we can restrict to controls $u \in \cU$ such that
\begin{equation}\label{eq:gpi0lesslambda}
g(\pi^{t,x,a,u}_0,\gamma^{t,a,u}_0) < \lambda.
\end{equation}

We define, for $\delta > 0$, the subdomains
\begin{align*}
Q^{\eta,\delta}_t &= \big\{x \in Q^\eta_t \, \big|\ d(x,\pa Q^\eta_t) > \delta\big\}, & B_\delta &= \big\{a \in \R^k \, \big|\ |a| < \delta^{-1}\big\},\\
S^m_\delta = Q^{\eta,\delta}_0 &= \big\{x \in S^m \, \big|\ d(x,\pa S^m) > \delta\big\}, & \cD^\eta_\delta &= \bigcup_{t \in [0,T]} \big(\{t\} \times Q^{\eta,\delta}_t \times B_\delta\big).
\end{align*}
As $g \in C^\uparrow(S^m \times \R^k;\R)$ (by part (iii) of Assumption~\ref{assumptions f g}), there exists a $\delta > 0$ sufficiently small such that
\begin{itemize}
\item $g > \lambda$ on $(S^m \times \R^k) \setminus (S^m_\delta \times B_\delta)$.
\end{itemize}
In particular, by \eqref{eq:gpi0lesslambda}, we have that $(\pi^{t,x,a,u}_0,\gamma^{t,a,u}_0) \in S^m_\delta \times B_\delta$ whenever $(t,x,a) \in \Delta$. Let $G \in C^\uparrow(S^m \times \R^k;\R)$ be a locally Lipschitz function, bounded below, and such that
\begin{itemize}
\item $G \leq g$,
\item $G = g$ on $S^m_\delta \times B_\delta$, and
\item $G > \lambda$ on $(S^m \times \R^k) \setminus (S^m_\delta \times B_\delta)$.
\end{itemize}
We define a new value function $\cK$, similarly to $\kappa^\eta$, but with initial cost function $G$, viz.
\begin{align*}
&\cK(t,x,a)\\
&= \inf_{u \in \cU} \bigg\{\int_0^t f(\pi^{t,x,a,u}_s,\gamma^{t,a,u}_s,u_s)\,\rd s + \int_0^t \psi(\pi^{t,x,a,u}_s,\gamma^{t,a,u}_s)\,\rd \eta_s + G(\pi^{t,x,a,u}_0,\gamma^{t,a,u}_0)\bigg\}.
\end{align*}
In particular, since $G \leq g$, $\cK$ is a viscosity subsolution of \eqref{eq:HMMsmoothHJB}--\eqref{eq:HMMsmoothHJBinitcond}. Moreover, by Proposition~\ref{prop kappa loc Lip x}, $\cK$ is locally Lipschitz continuous in $(x,a)$, uniformly in $t \in [0,T]$.

As noted above, we have that $(\pi^{t,x,a,u}_0,\gamma^{t,a,u}_0) \in S^m_\delta \times B_\delta$ for all $(t,x,a) \in \Delta$, which is also true for our new initial cost $G$, since $G > \lambda$ on $(S^m \times \R^k) \setminus (S^m_\delta \times B_\delta)$. As $G = g$ on $S^m_\delta \times B_\delta$, it follows that
\begin{equation}\label{eq:nu equals kappaeta}
\cK = \kappa^\eta \quad \text{on} \quad \Delta.
\end{equation}

The asymptotic growth rate of $\cK(t,x,a)$ as $d(x,\pa Q^\eta_t) \to 0$, and as $|a| \to \infty$, depends on the growth rate of the initial cost $G$. By taking $G$ to grow sufficiently slowly, we can ensure that $\cK$ is dominated asymptotically by $v$. More precisely, since $v \in C^\uparrow(\cD^\eta;\R)$, we can choose $G$ so that there exists an $\epsilon > 0$ such that
\begin{itemize}
\item $\Delta \subset \cD^\eta_\epsilon$, and
\item \begin{equation}\label{eq:cK less v on spatial boundary}
\cK \leq v \quad \text{on} \quad \bigcup_{t \in [0,T]} \Big(\{t\} \times \big((\pa Q^{\eta,\epsilon}_t \times B_\epsilon) \cup (Q^{\eta,\epsilon}_t \times \pa B_\epsilon)\big)\Big).
\end{equation}
\end{itemize}
As $\cK(0,x,a) = G(x,a) \leq g(x,a) \leq v(0,x,a)$ for all $(x,a) \in S^m \times \R^k$, we have in particular that
\begin{equation}\label{eq:cK less v on time boundary}
\cK \leq v \quad \text{on} \quad \{0\} \times Q^{\eta,\epsilon}_0 \times B_\epsilon.
\end{equation}
Combining \eqref{eq:cK less v on spatial boundary} and \eqref{eq:cK less v on time boundary}, we have that $\cK \leq v$ on the parabolic boundary of $\cD^\eta_\epsilon$. By the standard comparison principle for viscosity (sub/super-)solutions---e.g.~Theorem~5.1 in Barles\footnote{The hypotheses of the theorem are satisfied since $\cK$ is locally Lipschitz, and hence Lipschitz on the compact set $\overline{\cD}^\eta_\epsilon$. Strictly speaking the domain $\cD^\eta_\epsilon$ is not in general a simple Cartesian product of the spatial and temporal domains, as in the setting of Barles, but this is of no consequence, and requires only a trivial adaptation to the proof of the comparison principle.} \cite{AchdouBarlesLitvinovIshii2013}---applied with the subsolution $\cK$ and the supersolution $v$, it follows that $\cK \leq v$ on $\cD^\eta_\epsilon$. Since $\Delta \subset \cD^\eta_\epsilon$, it then follows from \eqref{eq:nu equals kappaeta} that
$$\kappa^\eta = \cK \leq v \quad \text{on} \quad \Delta.$$
Since $\Delta$ was arbitrary, we conclude that $\kappa^\eta \leq v$ on the entire domain $\cD^\eta$.
\end{proof}

\subsection{A rough HJ equation}\label{sec rough HJ eqn}

Replacing the smooth path $\eta$ by the rough path $\bY$ in \eqref{eq:HMMsmoothHJB}--\eqref{eq:HMMsmoothHJBinitcond}, we obtain the rough HJ equation:
\begin{align}
\rd\kappa + b\cdot\nabla_x\kappa\,\rd t + \sup_{u \in \R^k} \big\{u\cdot\nabla_a\kappa - f\big\}\hspace{0.8pt}\rd t + \big(\phi\cdot\nabla_x\kappa - \psi\big)\hspace{0.8pt}\rd \bY_t &= 0 & &\text{on}\hspace{10pt} \cD,\label{eq:HMMroughHJB}\\
\kappa(0,\cdot\hspace{1pt},\cdot) &= g & &\text{on}\hspace{10pt} S^m \times \R^k.\label{eq:HMMroughHJBinitcond}
\end{align}
We understand a solution to this equation in the sense of the following definition, sometimes known as a rough viscosity solution, used by Diehl et al.~\cite{DiehlFrizGassiat2017}, as well as for instance by Caruana, Friz and Oberhauser \cite{CaruanaFriz2009,CaruanaFrizOberhauser2011,FrizOberhauser2014} (see also Chapter~12 in Friz and Hairer \cite{FrizHairer2014}).

\begin{definition}\label{HMMdefnsolnroughHJB}
Given a smooth path $\eta$, we write $\bme = (\eta,\eta^{(2)}) \in \grp$ for its canonical lift, with $\eta^{(2)}$ defined as in \eqref{eq:HMM eta2 defn}. We write $\kappa^\eta$ for the unique viscosity solution of \eqref{eq:HMMsmoothHJB}--\eqref{eq:HMMsmoothHJBinitcond} in the class $C^\uparrow(\cD^\eta;\R)$, which by Theorem~\ref{theorem HMM kappaeta unique soln} is precisely the approximate value function, as defined in \eqref{eq:valuefuncsmoothHMM}. We say that a continuous function $v$ solves the rough HJ equation \eqref{eq:HMMroughHJB}--\eqref{eq:HMMroughHJBinitcond} if
\begin{equation}\label{eq:defn kappa eta n to v}
\kappa^{\eta^n} \longrightarrow\, v \qquad \text{as} \quad\ \ n\, \longrightarrow\, \infty
\end{equation}
locally uniformly on $\cD$, whenever $(\eta^n)_{n \geq 1}$ is a sequence of smooth paths such that $\bme^n \to \bY$ with respect to the $p$-variation rough path distance, i.e.~$\|\bme^n;\bY\|_p \to 0$ as $n \to \infty$.
\end{definition}

We note that if such a solution of \eqref{eq:HMMroughHJB}--\eqref{eq:HMMroughHJBinitcond} exists, then it is unique. Moreover, since the rough path $\bY \in \grp$ is geometric, there certainly exists such a sequence of smooth paths $(\eta^n)_{n \geq 1}$.

We can now state our main result.

\begin{theorem}\label{thm value func solves rough HJ}
Under Assumptions~\ref{assumption uncertainty parameterisation} and \ref{assumptions f g}, the value function $\kappa$, as defined in \eqref{eq:defnkappaHMM} and \eqref{eq:valuefuncHMM}, solves the rough HJ equation \eqref{eq:HMMroughHJB}--\eqref{eq:HMMroughHJBinitcond} in the sense of Definition~\ref{HMMdefnsolnroughHJB}. Moreover, writing $\kappa = \kappa^\bY$, the map from $\grp \to \R$ given by $\bY \mapsto \kappa^\bY(t,x,a)$ is locally uniformly continuous with respect to the $p$-variation rough path distance, locally uniformly in $(t,x,a)$.
\end{theorem}

\begin{proof}
Let $\bZ \in \rp$ be another rough path such that $\|\bY;\bZ\|_p \leq 1$. By possibly replacing $L$ by $L + 1$, we may assume that $\ver{\bZ}_p \leq L$. Let us write $\pi^{t,x,a,u,\bY}$ (resp.~$\pi^{t,x,a,u,\bZ}$) for the solution of the RDE \eqref{eq:contrdynamicsHMM} driven by $\bY$ (resp.~$\bZ$), and write $\kappa^\bY$ (resp.~$\kappa^\bZ$) for the corresponding value function, as defined in \eqref{eq:valuefuncHMM}, defined on the domain $\cD = \cD^\bY$ (resp.~$\cD^\bZ$).

Let $\Delta$ be a compact subset of $\cD^\bY \cap \cD^\bZ$. By Corollary~\ref{corollary HMM restrict controls}, we may restrict to controls $u \in \cU_{M,K}$, so that $\|\gamma^{t,a,u}\|_{\ptvarzt} \leq M$ for some $M > 0$, and there exists a compact subset $K \subset S^m \times \R^k$ such that $(\pi^{t,x,a,u,\bY}_s,\gamma^{t,a,u}_s) \in K$ and $(\pi^{t,x,a,u,\bZ}_s,\gamma^{t,a,u}_s) \in K$ for all $(t,x,a) \in \Delta$ and $s \in [0,t]$. We then deduce from Corollary~\ref{corollary reverse time}, that
\begin{align}
\|\pi^{t,x,a,u,\bY} - \pi^{t,x,a,u,\bZ}\|_{\pvarzt} &\lesssim \|\bY;\bZ\|_{\pvarzt},\label{eq:pi Y Z cts}\\
\bigg\|\int_0^\cdot \psi(\pi^{t,x,a,u,\bY}_s,\gamma^{t,a,u}_s)\,\rd\bY_s - \int_0^\cdot \psi(\pi^{t,x,a,u,\bZ}_s,\gamma^{t,a,u}_s)\,\rd\bZ_s\bigg\|_{\pvarzt} &\lesssim \|\bY;\bZ\|_{\pvarzt}.\nonumber
\end{align}
By part (i) of Assumption~\ref{assumptions f g}, as we have restricted the state trajectories to a compact set, we can take $f$ and $g$ to be Lipschitz in $(x,a)$, uniformly in $u$. Then, for any $(t,x,a) \in \Delta$,
\begin{align*}
\big|&\kappa^\bY(t,x,a) - \kappa^{\bZ}(t,x,a)\big|\\
&\leq \sup_{u \in \cU_{M,K}} \bigg|\int_0^t \big(f(\pi^{t,x,a,u,\bY}_s,\gamma^{t,a,u}_s,u_s) - f(\pi^{t,x,a,u,\bZ}_s,\gamma^{t,a,u}_s,u_s)\big)\,\rd s\\
&\hspace{50pt} + \int_0^t \psi(\pi^{t,x,a,u,\bY}_s,\gamma^{t,a,u}_s)\,\rd\bY_s - \int_0^t \psi(\pi^{t,x,a,u,\bZ}_s,\gamma^{t,a,u}_s)\,\rd\bZ_s\\
&\hspace{50pt} + g(\pi^{t,x,a,u,\bY}_0,\gamma^{t,a,u}_0) - g(\pi^{t,x,a,u,\bZ}_0,\gamma^{t,a,u}_0)\bigg|\\
&\lesssim \sup_{u \in \cU_{M,K}} \bigg(\int_0^t |\pi^{t,x,a,u,\bY}_s - \pi^{t,x,a,u,\bZ}_s|\,\rd s + \|\bY;\bZ\|_{\pvarzt} + |\pi^{t,x,a,u,\bY}_0 - \pi^{t,x,a,u,\bZ}_0|\bigg)\\
&\lesssim \|\bY;\bZ\|_{\pvarzt}.
\end{align*}

Let $(\eta^n)_{n \geq 1}$ be sequence of smooth paths such that $\|\bme^n;\bY\|_p \to 0$ as $n \to \infty$. It follows from \eqref{eq:pi Y Z cts} that the domain $\cD^{\eta_n}$ converges to $\cD$ in the obvious sense as $n \to \infty$. In particular, given a compact set $\Delta \subset \cD$, we then have that $\Delta \subset \cD^{\eta^n}$ for all sufficiently large $n$. The required convergence in \eqref{eq:defn kappa eta n to v} then follows by taking $\bZ = \bme^n$ in the above. The stated continuity of the value function with respect to the driving rough path is also immediate from the above.
\end{proof}

\begin{remark}
As shown in Theorem~\ref{thm value func solves rough HJ}, the value function $\kappa$ is continuous with respect to the enhanced observation path $\bY$. Although this does not in general imply continuity of the minimum point of the value function, in typical situations where $\kappa$ is convex and particularly unimodal (in the sense of having exactly one global minimum), this continuity with respect to $\bY$ will be inherited by the minimum point of $\kappa$, and hence by the most reasonable posterior and parameter value. Filters based on our approach are thus robust, both with respect to parameter uncertainty, and in the sense of continuity with respect to the (enhanced) observation path (cf.~Crisan et al.~\cite{CrisanDiehlFrizOberhauser2013}).
\end{remark}

\section{Numerical examples}\label{sec numerical examples}

\subsection{Unknown rate matrix}\label{subsec uncertain rate matrix}

As an example, let us take $m = 2$ and $d = 1$, so that the hidden signal $X$ is a 2-state Markov chain taking values in $\cX = \{e_1,e_2\}$, and the observation process $Y$ is 1-dimensional. We shall suppose that the observation vector $h$ is known and constant, but that the rate matrix $A$ depends on an unknown parameter $\lambda$, viz.
$$\cA = \bigg\{\bigg(\hspace{-4pt}\begin{array}{cc}
-\lambda & \nu - \lambda\\
\lambda & -(\nu - \lambda)
\end{array}\hspace{-4pt}\bigg)\, \bigg|\ \lambda \in (0,\nu)\bigg\} \qquad \text{and} \qquad \cH = \bigg\{\bigg(\hspace{-4pt}\begin{array}{c}
-\alpha\\
\alpha
\end{array}\hspace{-4pt}\bigg)\bigg\},$$
for some (known) $\nu, \alpha > 0$. In this case the dimension of the uncertain parameter is $k = 1$, and we adopt the smooth parametrisation:
\begin{equation*}
\R \ni a \, \longmapsto \, \frac{\nu}{1 + e^{-a}} = \lambda \in (0,\nu).
\end{equation*}
The observation process has dynamics:
$$\rd Y_t = \alpha\big(-\Ind_{\{X_t = e_1\}} + \Ind_{\{X_t = e_2\}}\big)\hspace{1pt}\rd t + \rd B_t.$$
Our objective is to determine the most reasonable values for the parameter $\lambda_t$ and for the posterior distribution of the signal $X_t$. We emphasize that here we do not assume any model for the dynamics of $\lambda$; we suppose that we only know that it lies in the interval $(0,\nu)$.

Recalling \eqref{eq:filtereqnStrat} and \eqref{eq:HMMfilterRDE}, we have
\begin{align*}
\rd \pi_{2,s} &= (\lambda_s - \nu\pi_{2,s})\,\rd s + 2\alpha\pi_{2,s}(1 - \pi_{2,s})\,\rd \bY_s\\
&=: b(\pi_{2,s},\gamma_s)\,\rd s + \phi(\pi_{2,s},\gamma_s)\,\rd \bY_s,
\end{align*}
where here $\pi = (\pi_1,\pi_2) = (1 - \pi_2,\pi_2)$ and $\gamma := \log(\lambda/(\nu - \lambda))$.

In this example we actually have $Q_t = S^2 \cong (0,1)$ for all $t \geq 0$, regardless of the realization of $\bY_{[0,t]}$. To see this, note that for $\pi_{2,s} \simeq 0$ we have $\rd \pi_{2,s} \simeq \lambda_s\,\rd s$, and since $\lambda$ takes values in $(0,\nu)$, we can take $\lambda_s$ sufficiently small (by choosing a suitable control $u$) to ensure that $\pi_{2,s} > 0$ for all $s < t$. Similarly, if $\pi_{2,s} \simeq 1$ we have $\rd \pi_{2,s} \simeq (\lambda_s - \nu)\,\rd s$, and we can take $\lambda_s$ sufficiently close to $\nu$ to ensure that $\pi_{2,s} < 1$ for all $s < t$.

We solve the problem in the variables $(q,\gamma)$, where we make the change of variables
$$(0,1) \times (0,\nu) \ni (\pi_2,\lambda) \, \longmapsto \, \bigg(\log\bigg(\frac{\pi_2}{1 - \pi_2}\bigg),\, \log\bigg(\frac{\lambda}{\nu - \lambda}\bigg)\bigg) =: (q,\gamma) \in \R^2.$$

We adopt a simple numerical approximation to solve the HJ equation \eqref{eq:HMMroughHJB}--\eqref{eq:HMMroughHJBinitcond}. We first recall that the most reasonable posteriors and parameter values are attained at the minimum point of the value function $\kappa$. To obtain an efficient scheme, we linearize the controlled dynamics around this minimum point (in the variables $(q,\gamma)$) and, taking the penalty functions $f$, $g$ to be quadratic, compute the resulting linear-quadratic optimization problem, linearizing around the new minimum point after each time step.

Recalling \eqref{eq:defn f psi}, we have
\begin{align}
\psi(\pi,\gamma) &= -h^{\hspace{-1pt}\top}\pi = \alpha(1 - 2\pi_2)\label{eq:numerical psi}\\
f(\pi,\gamma,u) &= \frf(\pi,\gamma,u) + \frac{1}{2}h^{\hspace{-1pt}\top}\hspace{-1pt}H\pi = \frf(\pi,\gamma,u) + \frac{\alpha^2}{2}.\label{eq:numerical f}
\end{align}
We choose the prior penalty $\frf$ to include the term $\frac{1}{2}|h^{\hspace{-1pt}\top}\pi|^2 = \frac{\alpha^2}{2}(2\pi_2 - 1)^2$, and introduce quadratic penalties to ensure that Assumption~\ref{assumptions f g} holds, viz.
\begin{align}
\tilde{f}(q,\gamma,u) &= \frac{\tau}{2}(q - \hat{q})^2 + \frac{\delta}{2}(\gamma - \hat{\gamma})^2 + \frac{1}{2\epsilon}u^2 + \frac{\alpha^2}{2}(2\pi_2 - 1)^2 + \frac{\alpha^2}{2},\label{eq:numerical tilde f}\\
\tilde{g}(q_0,\gamma_0) &= \frac{1}{2}\big(|q_0|^2 + |\gamma_0|^2\big)\label{eq:numerical tilde g}
\end{align}
for some constants $\tau, \delta, \epsilon > 0$.

We take $\nu = 1$, $\alpha = 1$, $\tau = \delta = 5 \times 10^{-2}$, $\epsilon = 10^{-3}$, and simulate the signal $X$ and observation process $Y$ with the `true' parameter $\lambda_t = 0.1 \times \Ind_{\{t \in [0,500)\}} + 0.7 \times \Ind_{\{t \in [500,1000)\}} + 0.3 \times \Ind_{\{t \in [1000,1500)\}} + 0.9 \times \Ind_{\{t \in [1500,2000]\}}$. We then compute the solution to the HJ equation using a standard Euler scheme with time steps of size $2 \times 10^{-3}$.

The learned value of $\gamma_t$ is given by \eqref{eq:argmin a kappa}, and the corresponding value of $\lambda_t$ is then obtained by reversing the above change of variables. The minimum point of the value function, which at each time $t$ corresponds to the most reasonable parameter value $\lambda_t$ given the observations up to time $t$, is compared with the true value of $\lambda_t$ in Figure~\ref{fig:lam plot}.

Similarly, reversing the above change of variables also gives the most reasonable posterior value, as in \eqref{eq:argmin x kappa}. This is then compared with the filter corresponding to the true parameter $\lambda$ in Figure~\ref{fig:post plot}.

\begin{figure}[!ht]
\centering
\input{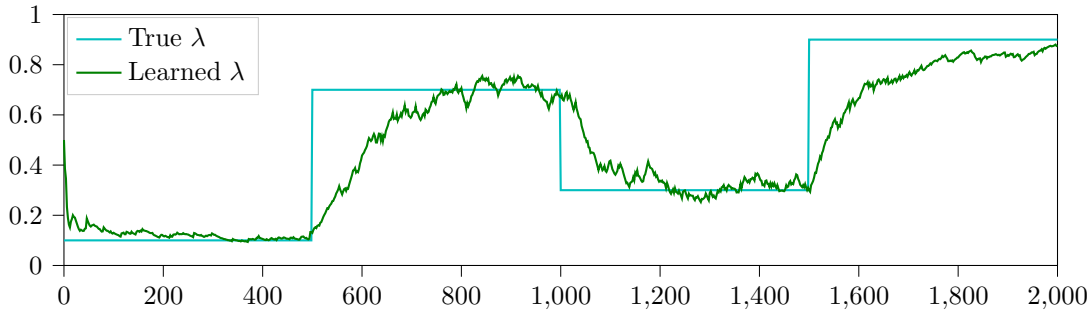}
\caption{Comparison of the most reasonable value of $\lambda_t$ with the true value.}
\label{fig:lam plot}
\end{figure}

\begin{figure}[!ht]
\centering
\input{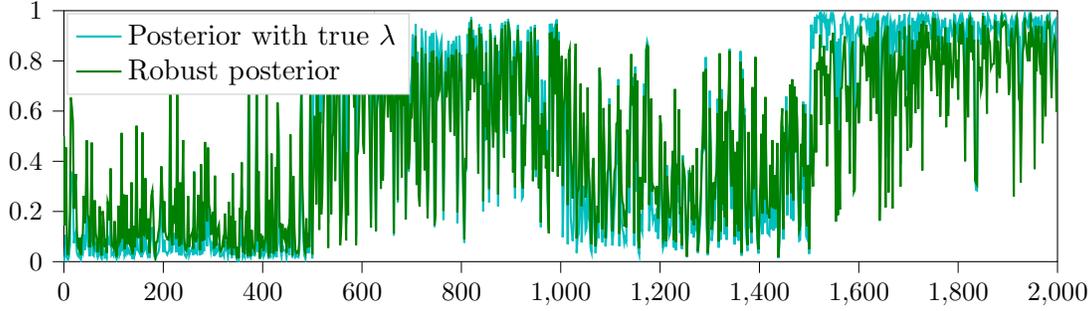}
\caption{Comparison of the most reasonable posterior with the posterior obtained using the true parameter $\lambda$.}
\label{fig:post plot}
\end{figure}

\subsection{Unknown observation matrix}

We now suppose that the rate matrix $A$ is known, but that the observation matrix $h$ depends on an unknown parameter $\alpha$, viz.
$$\cA = \bigg\{\bigg(\hspace{-4pt}\begin{array}{cc}
-\lambda & \mu\\
\lambda & -\mu
\end{array}\hspace{-4pt}\bigg)\bigg\} \qquad \text{and} \qquad \cH = \bigg\{\bigg(\hspace{-4pt}\begin{array}{c}
-\alpha\\
\alpha
\end{array}\hspace{-4pt}\bigg)\, \bigg|\ \alpha \in (\nu_1,\nu_2)\bigg\},$$
for some (known) $\lambda, \mu > 0$ and $0 \leq \nu_1 < \nu_2$. We adopt the parametrisation:
\begin{equation*}
\R \ni a \, \longmapsto \, \frac{\nu_2 + \nu_1e^{-a}}{1 + e^{-a}} = \alpha \in (\nu_1,\nu_2).
\end{equation*}
The observation process has dynamics:
\begin{equation*}
\rd Y_t = \alpha_t\big(-\Ind_{\{X_t = e_1\}} + \Ind_{\{X_t = e_2\}}\big)\hspace{1pt}\rd t + \rd B_t.
\end{equation*}
Note that if the jump rates $\lambda, \mu$ are sufficiently large then the filtering of the signal $X$ becomes an intractable task, as the observation time between jumps is too short to detect individual jumps. However, even in this case the problem of learning the unknown parameter $\alpha$ remains. We stress again that we do not assume any model for the dynamics of $\alpha$, assuming only that it takes values in the interval $(\nu_1,\nu_2)$.

We have
\begin{align*}
\rd \pi_{2,s} &= \big(\lambda(1 - \pi_{2,s}) - \mu\pi_{2,s}\big)\hspace{1pt}\rd s + 2\alpha_s\pi_{2,s}(1 - \pi_{2,s})\,\rd \bY_s\\
&=: b(\pi_{2,s},\gamma_s)\,\rd s + \phi(\pi_{2,s},\gamma_s)\,\rd \bY_s,
\end{align*}
where $\pi = (\pi_1,\pi_2) = (1 - \pi_2,\pi_2)$. As in the previous example, we solve the problem in the variables $(q,\gamma)$, where we make the change of variables
$$(0,1) \times (\nu_1,\nu_2) \ni (\pi_2,\alpha) \, \longmapsto \, \bigg(\log\bigg(\frac{\pi_2}{1 - \pi_2}\bigg),\, \log\bigg(\frac{\alpha - \nu_1}{\nu_2 - \alpha}\bigg)\bigg) =: (q,\gamma) \in \R^2.$$

As in the previous example, we adopt the penalty functions \eqref{eq:numerical psi}--\eqref{eq:numerical tilde g}. We take $\nu_1 = 0.2$, $\nu_2 = 1.8$, $\lambda = \mu = 5 \times 10^{-2}$, $\tau = \delta = 10^{-2}$, $\epsilon = 10^{-3}$, and simulate the signal $X$ and observation process $Y$ with the `true' parameter $\alpha = 0.4 \times \Ind_{\{t \in [0,500)\}} + 1.3 \times \Ind_{\{t \in [500,1000)\}} + 0.7 \times \Ind_{\{t \in [1000,1500)\}} + 1.6 \times \Ind_{\{t \in [1500,2000]\}}$. We then compute the solution to the HJ equation, using a linear-quadratic approximation scheme as in the previous example, with time steps of size $2 \times 10^{-3}$.

As noted above, due to the high jump rates of the signal, obtaining an accurate posterior value in this case is unfeasible. Despite this, it turns out that we can still obtain meaningful inference for the unknown parameter $\alpha$. The minimum point of the value function corresponds to the most reasonable value of $\alpha_t$, and is compared with the true parameter in Figure~\ref{fig:alpha plot}.

\begin{figure}[!ht]
\centering
\input{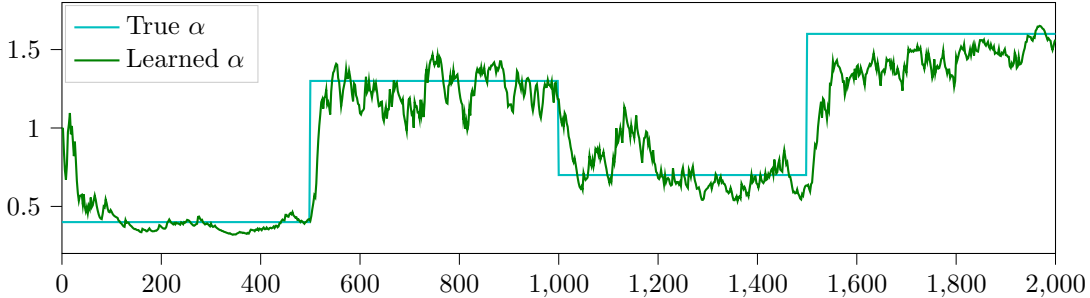}
\caption{Comparison of the most reasonable value of $\alpha_t$ with the true value.}
\label{fig:alpha plot}
\end{figure}

\section{Concluding remarks}\label{sec conclusion}

\begin{remark}
Since our evaluation of the `reasonability' of parameters is inherently non-probabilistic, our framework can be immediately extended to include non-Markovian models. Suppose for instance that a parameter $\gamma$ were, at each time $t$, a function of both $t$ and of the entire paths of the signal $X$ and observation process $Y$ up to time $t$, i.e.~$\gamma = \gamma(t,X_{\cdot \wedge t},Y_{\cdot \wedge t})$. Given a realization $\omega \in \Omega$, this then defines a path
$$t \, \longmapsto \, \gamma_t := \gamma(t,X_{\cdot \wedge t}(\omega),Y_{\cdot \wedge t}(\omega))$$
which may then be `learned' like any other deterministic path.
\end{remark}

\begin{remark}
In the current work we consider observations corrupted by Gaussian (Brownian) noise. More general observation noise with jumps could also be included in our theory, but would require further analysis. We expect that the recent results on c\`adl\`ag rough paths in Chevyrev and Friz \cite{ChevyrevFriz2019} or Friz and Zhang \cite{FrizZhang2018} could provide a suitable basis for the corresponding pathwise formulation.
\end{remark}

\begin{remark}
The setting of this paper is analogous to the dynamic generator, DR-expectation framework of Cohen \cite{Cohen2020}. One may alternatively consider the static generator case, where the unknown parameters are assumed to be constant in time. In this case formal calculations suggest that the value function $\kappa$ should satisfy an equation of the form
\begin{equation*}
\rd\kappa + \big(b\cdot\nabla_x\kappa - f\big)\hspace{0.8pt}\rd t + \big(\phi\cdot\nabla_x\kappa - \psi\big)\hspace{0.8pt}\rd \bY_t = 0
\end{equation*}
with $\kappa(0,\cdot\hspace{1pt},\cdot) = g$, which in principle must be solved separately for each $a \in \R^k$. This appears to be practically inconvenient, and it remains an open problem to derive some finite approximation or alternative approach which yields a tractable solution.

On the other hand, at least formally, the constant-parameter case is recovered in the present framework by imposing an infinite cost for non-zero controls $\dot{\gamma} = u$.
\end{remark}

\appendix

\section{Rough path estimates}\label{AppendixRoughPaths}

In this section we exhibit a direct approach to obtaining solutions to the RDE \eqref{eq:HMMRDE} in the controlled path setting of Gubinelli, based on the fixed point argument of Friz and Zhang \cite{FrizZhang2018}.

\subsection{Preliminary lemmas}

\begin{lemma}\label{lemmapvarpartitionbound}
For some $n \geq 1$, let $0 = t_0 < t_1 < \ldots < t_{n-1} < t_n = T$, be a partition of the interval $[0,T]$. Then, for any path $X$ and any $p \in [1,\infty)$, one has that
$$\|X\|_{\pvarzT} \leq n^{\frac{p - 1}{p}}\bigg(\sum_{i=1}^n\|X\|_{p,[t_{i-1},t_i]}^p\bigg)^{\hspace{-2pt}\frac{1}{p}}.$$
\end{lemma}

\begin{proof}
Let $0 = s_0 < s_1 < \ldots < s_{N-1} < s_N = T$ be another partition of the interval $[0,T]$. We can label the union of these two partitions in two different ways as follows. We can either write
\begin{align*}
s_{j-1} = t^j_0 < t^j_1 < \ldots < t^j_{n_j} = s_j \quad &\text{for each} \quad j = 1, \ldots, N,\\
\text{or} \quad t_{i-1} = s^i_0 < s^i_1 < \ldots < s^i_{N_i} = t_i \quad &\text{for each} \quad i = 1, \ldots, n,
\end{align*}
where in particular $n_j \leq n$ for every $j$. For every $j = 1,\ldots,N$, we have, by H\"older's inequality,
$$\sum_{i=1}^{n_j}\big|X_{t^j_i} - X_{t^j_{i-1}}\hspace{-1pt}\big| \leq n_j^{\frac{p - 1}{p}} \bigg(\sum_{i=1}^{n_j}\big|X_{t^j_i} - X_{t^j_{i-1}}\hspace{-1pt}\big|^p\bigg)^{\hspace{-2pt}\frac{1}{p}} \leq n^{\frac{p - 1}{p}} \bigg(\sum_{i=1}^{n_j}\big|X_{t^j_i} - X_{t^j_{i-1}}\hspace{-1pt}\big|^p\bigg)^{\hspace{-2pt}\frac{1}{p}}.$$
Then
\begin{align*}
\sum_{j=1}^N\big|X_{s_j} - X_{s_{j-1}}\hspace{-1pt}\big|^p &\leq \sum_{j=1}^N\bigg(\sum_{i=1}^{n_j}\big|X_{t^j_i} - X_{t^j_{i-1}}\hspace{-1pt}\big|\bigg)^{\hspace{-2pt}p} \leq n^{p - 1}\sum_{j=1}^N\sum_{i=1}^{n_j}\big|X_{t^j_i} - X_{t^j_{i-1}}\hspace{-1pt}\big|^p\\
&= n^{p - 1}\sum_{i=1}^{n}\sum_{j=1}^{N_i}\big|X_{s^i_j} - X_{s^i_{j-1}}\hspace{-1pt}\big|^p \leq n^{p - 1}\sum_{i=1}^{n}\|X\|_{p,[t_{i-1},t_i]}^p.
\end{align*}
The result then follows from taking the supremum over all possible partitions $s_0 < s_1 < \ldots < s_N$ of the interval $[0,T]$.
\end{proof}

\begin{lemma}\label{lemmainvarianceHMM}
Let $\phi \in C^3_b$, $\gamma \in \Cptvar$, and let $\bY = (Y,\Y) \in \rp$ with $\ver{\bY}_p \leq L$ for some $L > 0$. For any controlled path $(X,X') \in \crp$, we have that
$$\bigg(\int_0^\cdot\phi(X_r,\gamma_r)\,\rd\bY_r,\, \phi(X,\gamma)\bigg) \in \crp,$$
and the estimates
\begin{align*}
\|\phi(X,\gamma)\|_p &\leq C\Big(\big(|X'_0| + \|X'\|_p\big)\|Y\|_p + \|R^X\|_{\frac{p}{2}} + \|\gamma\|_{\frac{p}{2}}\Big),\\
\big\|R^{\int_0^\cdot\phi(X_r,\gamma_r)\,\rd\bY_r}\big\|_{\frac{p}{2}} &\leq C\Big(1 + |X'_0| + \|X'\|_p + \|R^X\|_{\frac{p}{2}} + \|\gamma\|_{\frac{p}{2}}\Big)^{\hspace{-2pt}2}\ver{\bY}_p,
\end{align*}
where the constant $C$ depends only on $\phi, p$ and $L$.
\end{lemma}

\begin{proof}
Note that $\gamma \in \Cptvar$ immediately satisfies \eqref{eq:Gub deriv} with $\gamma'_s = 0$ and $R^\gamma_{s,t} = \gamma_{s,t}$. It is then straightforward to see that the composition $\phi(X,\gamma)$ is then itself a controlled path with Gubinelli derivative $\phi(X,\gamma)' = D_x\phi(X,\gamma)X'$, where $D_x\phi$ denotes the Fr\'echet derivative of $\phi$ in its first argument. From Lemma~3.6 in \cite{FrizZhang2018}, we immediately have
\begin{align*}
\|\phi(X,\gamma)\|_p &\leq C\Big(\big(|X'_0| + |\gamma'_0| + \|X'\|_p + \|\gamma'\|_p\big)\|Y\|_p + \|R^X\|_{\frac{p}{2}} + \|R^\gamma\|_{\frac{p}{2}}\Big),\\
\big\|R^{\int_0^\cdot\phi(X_r,\gamma_r)\,\rd\bY_r}\big\|_{\frac{p}{2}} &\leq C\Big(1 + |X'_0| + |\gamma'_0| + \|X'\|_p + \|\gamma'\|_p + \|R^X\|_{\frac{p}{2}} + \|R^\gamma\|_{\frac{p}{2}}\Big)^{\hspace{-2pt}2}\ver{\bY}_p,
\end{align*}
for some constant $C$ depending only on $\phi$, $p$ and $\|Y\|_p$. Substituting $\gamma'_s = 0$ and $R^\gamma_{s,t} = \gamma_{s,t}$ into the above, we obtain the desired estimates.
\end{proof}

The following lemma follows similarly from \cite[Lemma~3.7]{FrizZhang2018}.

\begin{lemma}\label{lemmacontractionHMM}
Let $\phi \in C^3_b$, $\gamma, \tg \in \Cptvar$, and let $\bY = (Y,\Y)$ and $\tbY = (\tY,\teY)$ be rough paths with $\ver{\bY}_p, \ver{\tbY}_p \leq L$ for some $L > 0$. Suppose that $(X,X') \in \crp$ and $(\tX,\tX') \in \crpt$ satisfy $|X'_0| + \|X'\|_p + \|R^X\|_{\frac{p}{2}} + \|\gamma\|_{\frac{p}{2}} \leq M$ and $|\tX'_0| + \|\tX'\|_p + \|R^{\tX}\|_{\frac{p}{2}} + \|\tg\|_{\frac{p}{2}} \leq M$ for some $M > 0$. Then we have the estimates
\begin{align*}
\big\|\phi(X,\gamma) - \phi(\tX,\tg)\big\|_p \leq C\Big(&|X_0 - \tX_0| + |X'_0 - \tX'_0| + \|X' - \tX'\|_p\|Y\|_p\\
&+ \|R^X - R^{\tX}\|_{\frac{p}{2}} + |\gamma_0 - \tg_0| + \|\gamma - \tg\|_{\frac{p}{2}} + \|Y - \tY\|_p\Big),
\end{align*}
\begin{align*}
\big\|&R^{\int_0^\cdot\phi(X_r,\gamma_r)\,\rd\bY_r} - R^{\int_0^\cdot\phi(\tX_r,\tg_r)\,\rd\tbY_r}\big\|_{\frac{p}{2}}\\
&\leq C\Big(\|\bY;\tbY\|_p + \Big(|X_0 - \tX_0| + |X'_0 - \tX'_0| + \|X' - \tX'\|_p\\
&\hspace{100pt} + \|R^X - R^{\tX}\|_{\frac{p}{2}} + |\gamma_0 - \tg_0| + \|\gamma - \tg\|_{\frac{p}{2}}\Big)\ver{\bY}_p\Big),
\end{align*}
where the constant $C$ depends on $\phi, p, L$ and $M$.
\end{lemma}

\subsection{Proof of Theorem~\ref{theoremHMMRDE}}

\begin{proof}[Proof of Theorem~\ref{theoremHMMRDE}]
\textbf{Step 1.}
We define the map $\cM_t \colon \crp([0,t];\R^m) \to \crp([0,t];\R^m)$ by
$$\cM_t(X,X') = \bigg(x + \int_0^\cdot b(X_r,\gamma_r)\,\rd r + \int_0^\cdot\phi(X_r,\gamma_r)\,\rd\bY_r,\ \phi(X,\gamma)\bigg),$$
and, for $\delta \geq 1$, introduce the subset of controlled paths
$$\cB^{(\delta)}_t := \Big\{(X,X') \in \crp([0,t];\R^m) : (X_0,X'_0) = (x,\phi(x,\gamma_0)),\ \|X,X'\|_{Y,p}^{(\delta)} \leq 1\Big\},$$
where
$$\|X,X'\|_{Y,p}^{(\delta)} := \|X'\|_p + \delta\|R^X\|_{\frac{p}{2}}.$$
Let $L > 0$ such that $\ver{\bY}_p \leq L$. Provided that
\begin{equation}\label{eq:gammaleq1}
\|\gamma\|_{\ptvarzt} \leq 1,
\end{equation}
by Lemma~\ref{lemmainvarianceHMM}, for any $(X,X') \in \cB^{(\delta)}_t$, we have
\begin{align*}
\|&\cM_t(X,X')\|_{Y,p}^{(\delta)}\\
&\leq \|\phi(X,\gamma)\|_{\pvarzt} + \delta\big\|R^{\int_0^\cdot\phi(X_r,\gamma_r)\,\rd\bY_r}\big\|_{\ptvarzt} + \delta\bigg\|\int_0^\cdot b(X_r,\gamma_r)\,\rd r\bigg\|_{\ptvarzt}\\
&\leq C_1\bigg(\frac{1}{\delta} + \delta\big(\ver{\bY}_{\pvarzt} + t + \|\gamma\|_{\ptvarzt}\big)\bigg)
\end{align*}
for some constant $C_1 \geq \frac{1}{2}$ depending only on $b, \phi, p$ and $L$. Let $\delta = \delta_1 := 2C_1 \geq 1$, so that
$$\|\cM_t(X,X')\|_{Y,p}^{(\delta_1)} \leq \frac{1}{2} + 2C_1^2\big(\ver{\bY}_{\pvarzt} + t + \|\gamma\|_{\ptvarzt}\big).$$
By taking $t = t_1$ sufficiently small such that
\begin{equation}\label{eq:smalltimet1}
\ver{\bY}_{p,[0,t_1]} + t_1 + \|\gamma\|_{\frac{p}{2},[0,t_1]} \leq \frac{1}{4C_1^2},
\end{equation}
which in particular ensures that \eqref{eq:gammaleq1} holds, we obtain $\|\cM_{t_1}(X,X')\|_{Y,p}^{(\delta_1)} \leq 1$. That is, $\cB^{(\delta_1)}_{t_1}$ is invariant under $\cM_{t_1}$.

\textbf{Step 2.}
Let $(X,X'), (\tX,\tX') \in \cB^{(\delta_1)}_t$ with $\delta_1$ as above and some $t \leq t_1$. Note that $|X'_0| + \|X'\|_p + \|R^X\|_{\frac{p}{2}} + \|\gamma\|_{\frac{p}{2}} \leq |\phi(x,\gamma_0)| + 2$ and similarly for $\tX$, so that the hypotheses of Lemma~\ref{lemmacontractionHMM} are satisfied. For any (new) $\delta \geq 1$, we then have
\begin{align*}
\big\|&\cM_t(X,X') - \cM_t(\tX,\tX')\big\|_{Y,p}^{(\delta)}\\
&\leq \big\|\phi(X,\gamma) - \phi(\tX,\gamma)\big\|_p + \delta\big\|R^{\int_0^\cdot\phi(X_r,\gamma_r)\,\rd\bY_r} - R^{\int_0^\cdot\phi(\tX_r,\gamma_r)\,\rd\bY_r}\big\|_{\frac{p}{2}}\\
&\qquad + \delta\bigg\|\int_0^\cdot b(X_r,\gamma_r)\,\rd r - \int_0^\cdot b(\tX_r,\gamma_r)\,\rd r\bigg\|_{\frac{p}{2}}\\
&\leq C_2\Big(\|R^X - R^{\tX}\|_{\frac{p}{2}} + \delta\Big(\|X' - \tX'\|_p + \|R^X - R^{\tX}\|_{\frac{p}{2}}\Big)\big(\ver{\bY}_p + t\big)\Big),
\end{align*}
for some constant $C_2 > \frac{1}{2}$ depending only on $b, \phi, p$ and $L$, where in bounding the drift terms we used the fact\footnote{This follows by taking the difference of the equalities $X_{0,t} = X'_0Y_{0,t} + R^X_{0,t}$ and $\tX_{0,t} = \tX'_0Y_{0,t} + R^{\tX}_{0,t}$.} that $\|X - \tX\|_\infty \leq \|R^X - R^{\tX}\|_{\frac{p}{2}}$. Let $\delta = \delta_2 := 2C_2 > 1$, so that
\begin{align*}
\big\|\cM_t(X,X') - \cM_t(\tX,\tX'&)\big\|_{Y,p}^{(\delta_2)} \leq \frac{\delta_2}{2}\|R^X - R^{\tX}\|_{\frac{p}{2}}\\
&+ 2C_2^2\Big(\|X' - \tX'\|_p + \|R^X - R^{\tX}\|_{\frac{p}{2}}\Big)\big(\ver{\bY}_{\pvarzt} + t\big).
\end{align*}
Taking $t = t_2 \leq t_1$ sufficiently small such that
\begin{equation}\label{eq:smalltimet2}
\ver{\bY}_{p,[0,t_2]} + t_2 \leq \frac{1}{4C_2^2},
\end{equation}
we obtain
\begin{align*}
\big\|\cM_{t_2}(X,X') - \cM_{t_2}(\tX,\tX')\big\|_{Y,p}^{(\delta_2)} &\leq \frac{1}{2}\|X' - \tX'\|_p + \frac{\delta_2 + 1}{2}\|R^X - R^{\tX}\|_{\frac{p}{2}}\\
&\leq \frac{\delta_2 + 1}{2\delta_2}\big\|(X,X') - (\tX,\tX')\big\|_{Y,p}^{(\delta_2)},
\end{align*}
and we have thus established that $\cM_{t_2}$ is a contraction on $\cB^{(\delta_1)}_{t_2}$ with respect to the norm $\|\cdot\|_{Y,p}^{(\delta_2)}$. The fixed point of this map is then the unique solution of the RDE \eqref{eq:HMMRDE} with initial condition $X_0 = x$ over the time interval $[0,t_2]$ satisfying $X' = \phi(X,\gamma)$.

\textbf{Step 3.}
Let $\tbY = (\tY,\teY) \in \rp$ be another rough path with $\ver{\tbY}_p \leq L$, let $\tg \in \Cptvar$ and $\tilde{x} \in \R^m$. Let $(X,X') = (X,\phi(X,\gamma)) \in \crp$ (resp.~$(\tX,\tX') = (\tX,\phi(\tX,\tg)) \in \crpt$) be the unique solution of the RDE \eqref{eq:HMMRDE} driven by $\bY$ (resp.~$\tbY$) with parameter $\gamma$ (resp.~$\tg$) and initial value $x$ (resp.~$\tilde{x}$), which exists over the time interval $[0,t_2]$ (resp.~$[0,\tilde{t}_2]$) by the above. In the following we denote by $\lesssim$ inequality up to a multiplicative constant which may depend on $b, \phi, p$ and $L$.

From the equality $X_{s,t} - \tX_{s,t} = X'_sY_{s,t} - \tX'_s\tY_{s,t} + R^X_{s,t} - R^{\tX}_{s,t}$, we deduce
\begin{align}
\|X - \tX\|_p &\leq \|X' - \tX'\|_\infty\|Y\|_p + \|\phi(\tX,\tg)\|_\infty\|Y - \tY\|_p + \|R^X - R^{\tX}\|_{\frac{p}{2}}\nonumber\\
&\lesssim |X_0 - \tX_0| + |\gamma_0 - \tg_0| + \|X' - \tX'\|_p + \|Y - \tY\|_p + \|R^X - R^{\tX}\|_{\frac{p}{2}}.\label{eq:XtXpest}
\end{align}
By Lemma~\ref{lemmacontractionHMM}, for any (new) $\delta \geq 1$, we have
\begin{align*}
&\|X' - \tX'\|_p + \delta\|R^X - R^{\tX}\|_{\frac{p}{2}}\\
&\leq \big\|\phi(X,\gamma) - \phi(\tX,\tg)\big\|_p + \delta\big\|R^{\int_0^\cdot\phi(X_r,\gamma_r)\,\rd\bY_r} - R^{\int_0^\cdot\phi(\tX_r,\tg_r)\,\rd\tbY_r}\big\|_{\frac{p}{2}}\\
&\quad + \delta\bigg\|\int_0^\cdot b(X_r,\gamma_r)\big)\,\rd r - \int_0^\cdot b(\tX_r,\tg_r)\big)\,\rd r\bigg\|_{\frac{p}{2}}\\
&\lesssim |X_0 - \tX_0| + \|R^X - R^{\tX}\|_{\frac{p}{2}} + |\gamma_0 - \tg_0| + \|\gamma - \tg\|_{\frac{p}{2}}\\
&\quad + \delta\Big(\Big(|X_0 - \tX_0| + \|X' - \tX'\|_p + \|R^X - R^{\tX}\|_{\frac{p}{2}} + |\gamma_0 - \tg_0| + \|\gamma - \tg\|_{\frac{p}{2}}\Big)\ver{\bY}_p\\
&\hspace{40pt} + \|\bY;\tbY\|_p + \big(\|X - \tX\|_\infty + \|\gamma - \tg\|_\infty\big)t\Big).
\end{align*}
Writing $\|X - \tX\|_\infty \leq |X_0 - \tX_0| + \|X - \tX\|_p$ and using \eqref{eq:XtXpest}, we then have
\begin{align*}
\|&X' - \tX'\|_p + \delta\|R^X - R^{\tX}\|_{\frac{p}{2}}\\
&\leq C_3\Big(|X_0 - \tX_0| + \|R^X - R^{\tX}\|_{\frac{p}{2}} + |\gamma_0 - \tg_0| + \|\gamma - \tg\|_{\frac{p}{2}} + \delta\|\bY;\tbY\|_p\\
&\hspace{37pt} + \delta\Big(|X_0 - \tX_0| + \|X' - \tX'\|_p + \|R^X - R^{\tX}\|_{\frac{p}{2}}\\
&\hspace{65pt} + \|Y - \tY\|_p + |\gamma_0 - \tg_0| + \|\gamma - \tg\|_{\frac{p}{2}}\Big)\big(\ver{\bY}_{\pvarzt} + t\big)\Big)
\end{align*}
for some constant $C_3 > \frac{1}{2}$ depending only on $b, \phi, p$ and $L$. Let $\delta = \delta_3 := 2C_3 > 1$. Taking $t = t_3 \leq t_2 \wedge \tilde{t}_2$ sufficiently small such that
\begin{equation}\label{eq:smalltimet3}
\ver{\bY}_{p,[0,t_3]} + t_3 \leq \frac{1}{4C_3^2}
\end{equation}
and rearranging, we obtain
\begin{align}
\|&X' - \tX'\|_{p,[0,t_3]} + (\delta_3 - 1)\|R^X - R^{\tX}\|_{\frac{p}{2},[0,t_3]}\nonumber\\
&\leq (\delta_3 + 1)\Big(|X_0 - \tX_0| + |\gamma_0 - \tg_0| + \|\gamma - \tg\|_{\frac{p}{2},[0,t_3]}\Big) + (\delta_3^2 + 1)\|\bY;\tbY\|_{p,[0,t_3]}.\label{eq:RDEcontXdRX}
\end{align}
Combining this with \eqref{eq:XtXpest}, we obtain
\begin{equation}\label{eq:RDEcontX}
\|X - \tX\|_{p,[0,t_3]} \lesssim |X_0 - \tX_0| + \|\bY;\tbY\|_{p,[0,t_3]} + |\gamma_0 - \tg_0| + \|\gamma - \tg\|_{\frac{p}{2},[0,t_3]}.
\end{equation}

\textbf{Step 4.}
From now on we allow the multiplicative constant indicated by $\lesssim$ to also depend on $\psi$. We have
\begin{align}
&\bigg\|\int_0^\cdot\psi(X_r,\gamma_r)\,\rd\bY_r - \int_0^\cdot\psi(\tX_r,\tg_r)\,\rd\tbY_r\bigg\|_{p,[0,t_3]}\nonumber\\
&\lesssim \big(\|X - \tX\|_\infty + \|\gamma - \tg\|_\infty\big)\|Y\|_{p,[0,t_3]} + \|Y - \tY\|_{p,[0,t_3]}\nonumber\\
&\qquad + \big\|R^{\int_0^\cdot\psi(X_r)\,\rd\bY_r} - R^{\int_0^\cdot\psi(\tX_r)\,\rd\tbY_r}\big\|_{\frac{p}{2},[0,t_3]}\nonumber\\
&\lesssim |X_0 - \tX_0| + \|X - \tX\|_{p,[0,t_3]} + \|X' - \tX'\|_{p,[0,t_3]} + \|R^X - R^{\tX}\|_{\frac{p}{2},[0,t_3]}\nonumber\\
&\qquad + \|\bY;\tbY\|_{p,[0,t_3]} + |\gamma_0 - \tg_0| + \|\gamma - \tg\|_{\frac{p}{2},[0,t_3]}\nonumber\\
&\lesssim |X_0 - \tX_0| + \|\bY;\tbY\|_{p,[0,t_3]} + |\gamma_0 - \tg_0| + \|\gamma - \tg\|_{\frac{p}{2},[0,t_3]},\label{eq:RDEcontintpsi}
\end{align}
where we have applied Lemma~\ref{lemmacontractionHMM}, and then used \eqref{eq:RDEcontXdRX} and \eqref{eq:RDEcontX} to obtain the last line.

\textbf{Step 5.}
Recalling \eqref{eq:smalltimet1} and \eqref{eq:smalltimet2}, we have shown that on any time interval $[s,t]$ sufficiently small that
\begin{equation}\label{eq:smalltimeh}
\ver{\bY}_{p,[s,t]} + |t - s| + \|\gamma\|_{\frac{p}{2},[s,t]} \leq h := \min\bigg\{\frac{1}{4C_1^2},\frac{1}{4C_2^2}\bigg\},
\end{equation}
given an initial value $X_s$, there exists a unique solution $X$ to the RDE \eqref{eq:HMMRDE} over the interval $[s,t]$ corresponding to the driving rough path $\bY$ and the parameter $\gamma$. Moreover, if \eqref{eq:smalltimet1}, \eqref{eq:smalltimet2} and additionally \eqref{eq:smalltimet3} hold for both ($\bY, \gamma$) and ($\tbY, \tg$), then the estimates \eqref{eq:RDEcontX} and \eqref{eq:RDEcontintpsi} also hold over the interval $[s,t]$.

We can divide the (arbitrary) time interval $[0,T]$ into a partition $0 = s_0 < s_1 < \ldots < s_N = T$ such that \eqref{eq:smalltimeh} holds on each subinterval $[s_{i-1},s_i]$, $i = 1,\ldots,N$. That this partition is finite follows from the compactness of the interval $[0,T]$.

Note that, crucially, all of our estimates hold independently of the initial values $X_0 = x$ and $\gamma_0$. Thus, by pasting solutions on each of these subintervals together, we obtain a unique global solution $X$, which holds over the entire interval $[0,T]$.

By possibly increasing the length of each subinterval, we may assume without loss of generality that the inequality in \eqref{eq:smalltimeh} is actually an equality on every subinterval $[s_{i-1},s_i]$ in our partition (except possibly on one such subinterval). It then follows that, on each subinterval $[s_{i-1},s_i]$, at least one of $\|Y\|_{p,[s_{i-1},s_i]} \geq h/4$ or $\|\Y\|_{\frac{p}{2},[s_{i-1},s_i]} \geq h/4$ or $|s_i - s_{i-1}| \geq h/4$ or $\|\gamma\|_{\frac{p}{2},[s_{i-1},s_i]} \geq h/4$ holds. Since
\begin{gather*}
\sum_{i=1}^n \|Y\|_{p,[s_{i-1},s_i]}^p \leq \|Y\|_{p,[0,T]}^p, \qquad \quad \sum_{i=1}^n \|\Y\|_{\frac{p}{2},[s_{i-1},s_i]}^{\frac{p}{2}} \leq \|\Y\|_{\frac{p}{2},[0,T]}^{\frac{p}{2}},\\
\sum_{i=1}^n |s_i - s_{i-1}| \leq T, \qquad \text{and} \qquad \sum_{i=1}^n \|\gamma\|_{\frac{p}{2},[s_{i-1},s_i]}^{\frac{p}{2}} \leq \|\gamma\|_{\frac{p}{2},[0,T]}^{\frac{p}{2}}
\end{gather*}
for every $n \geq 1$, we infer an upper bound on the total number of subintervals:
\begin{equation}\label{eq:boundonnosubintsN}
N \lesssim \|Y\|_{p,[0,T]}^p + \|\Y\|_{\frac{p}{2},[0,T]}^{\frac{p}{2}} + T + \|\gamma\|_{\frac{p}{2},[0,T]}^{\frac{p}{2}} \lesssim 1 + T + \|\gamma\|_{\frac{p}{2},[0,T]}^{\frac{p}{2}}.
\end{equation}

It follows from Lemma~\ref{lemmainvarianceHMM} that
$$\bigg\|\int_0^\cdot\psi(X_r,\gamma_r)\,\rd\bY_r\bigg\|_{p,[s_{i-1},s_i]} \lesssim \ver{\bY}_{p,[s_{i-1},s_i]}.$$
Hence, by Lemma~\ref{lemmapvarpartitionbound},
\begin{align*}
\bigg\|\int_0^\cdot\psi(X_r,\gamma_r)\,\rd\bY_r\bigg\|_{p,[0,T]} &\leq N^{\frac{p - 1}{p}} \bigg(\sum_{i=1}^N \bigg\|\int_0^\cdot\psi(X_r,\gamma_r)\,\rd\bY_r\bigg\|_{p,[s_{i-1},s_i]}^p\bigg)^{\hspace{-2.5pt}\frac{1}{p}}\\
&\lesssim N^{\frac{p - 1}{p}} \ver{\bY}_{p,[0,T]}
\end{align*}
and, combining this with \eqref{eq:boundonnosubintsN}, we obtain \eqref{eq:HMMroughintpsibound}.

We now choose the partition $0 = s_0 < s_1 < \ldots < s_N = T$ so that \eqref{eq:smalltimet1}, \eqref{eq:smalltimet2} and \eqref{eq:smalltimet3} all hold on each subinterval $[s_{i-1},s_i]$ for both ($\bY, \gamma$) and ($\tbY, \tg$). Since
$$\|X - \tX\|_{p,[0,T]} \leq N^{\frac{p - 1}{p}} \bigg(\sum_{i=1}^N \|X - \tX\|_{p,[s_{i-1},s_i]}^p\bigg)^{\hspace{-2pt}\frac{1}{p}}$$
(again by Lemma~\ref{lemmapvarpartitionbound}), we can paste the local estimates \eqref{eq:RDEcontX}--\eqref{eq:RDEcontintpsi} over each subinterval to obtain the global estimates \eqref{eq:RDEcontXglobal}--\eqref{eq:RDEcontintpsiglobal}.

The statements on consistency of rough and stochastic integrals and differential equations are standard; see for example Corollary~5.2 and Theorem~9.1 in \cite{FrizHairer2014}.
\end{proof}

\subsection{Proofs of supplementary results}

\begin{proof}[Proof of Corollary~\ref{corollary reverse time}]
Setting $\hat{Y}_t = Y_{T-t}$ and $\hat{\Y}_{s,t} = -\Y_{T-t,T-s} + Y_{T-t,T-s} \otimes Y_{T-t,T-s}$, we obtain a rough path $\hat{\bY} = (\hat{Y},\hat{\Y}) \in \rp$ which is in fact the time reversal of the original path $\bY$. Letting $\hat{\gamma}_t = \gamma_{T-t}$, the solution $\hat{X}$ of the RDE \eqref{eq:HMMRDE} (with $\rd t$ replaced by $-\rd t$) driven by $\hat{\bY}$ with parameter $\hat{\gamma}$ and initial condition $\hat{X}_0 = X_T$, is then simply the time reversal of the solution $X$ of the original RDE. By applying part (iii) of Theorem~\ref{theoremHMMRDE} to the corresponding time reversals of the solutions $X$, $\tX$, we deduce the desired estimates.
\end{proof}

\begin{proof}[Proof of Lemma~\ref{lemma sharp rough bound}]
Fix $p \in [2,3)$ and $q < \frac{p - 1}{2}$. Let $\vp \colon \R \to \R$ be the $4$-periodic function satisfying
\begin{equation*}
\vp(s) = \left\{\begin{array}{cc}
s &\quad \text{for}\quad\ 0 \leq s \leq 1,\\
1 &\quad \text{for}\quad\ 1 \leq s \leq 2,\\
3 - s &\quad \text{for}\quad\ 2 \leq s \leq 3,\\
0 &\quad \text{for}\quad\ 3 \leq s \leq 4.
\end{array}\right.
\end{equation*}
Let $T = 4$, and $\epsilon \in (0,\frac{2}{p})$. For integer $n \geq 1$, let
$$Y^n_t = (2n)^{-\frac{1}{p}}\vp(nt) \quad \text{and} \quad \gamma^n_t = 2^{-\frac{2}{p}}n^{-\epsilon}\vp(nt + 1) \quad \text{for} \quad t \in [0,T],$$
and $\bY^n = (Y^n,\Y^n)$, where $\Y^n$ is the canonical enhancement of $Y^n$, so that
$$\Y^n_{s,t} = \int_s^t Y^n_{s,r}\,\rd Y^n_r = \frac{1}{2}(Y^n_{s,t})^2, \qquad \text{for} \quad (s,t) \in \simplex.$$
Let $\psi \colon \R \to \R$ be a $C^3_b$ function which is equal to the identity on the interval $[0,1]$. We then have that $\|Y^n\|_{p,[0,T]} = 1$, $\|\Y^n\|_{\frac{p}{2},[0,T]} = \frac{1}{2}$, $\|\gamma^n\|_{\frac{p}{2},[0,T]} = n^{\frac{2}{p} - \epsilon}$, and
$$\int_0^T \psi(\gamma^n_r)\,\rd Y^n_r = \int_0^T \gamma^n_r\,\rd Y^n_r = 2^{-\frac{3}{p}}n^{1 - \frac{1}{p} - \epsilon}.$$

Since $q < \frac{p - 1}{2}$, we may take $\epsilon \in (0,\frac{2}{p})$ sufficiently small such that
$$\frac{p - 1 - \epsilon p}{2 - \epsilon p} > q,$$
and rearranging then gives
$$1 - \frac{1}{p} - \epsilon > q\bigg(\frac{2}{p} - \epsilon\bigg).$$
It follows that there does not exist a constant $C$ such that
$$\int_0^T \psi(\gamma^n_r)\,\rd Y^n_r = 2^{-\frac{3}{p}}n^{1 - \frac{1}{p} - \epsilon} \leq C\big(1 + n^{q(\frac{2}{p} - \epsilon)}\big)\frac{3}{2} = C\big(1 + \|\gamma^n\|_{\frac{p}{2},[0,T]}^{q}\big)\ver{\bY^n}_{p,[0,T]}$$
holds for every $n \geq 1$.
\end{proof}

\bibliographystyle{abbrv}
\bibliography{References_Andy}

\begin{thebibliography}{10}

\bibitem{AchdouBarlesLitvinovIshii2013}
Y.~Achdou, G.~Barles, H.~Ishii, and G.~L. Litvinov.
\newblock {\em Hamilton--Jacobi Equations: Approximations, Numerical Analysis
  and Applications}.
\newblock Springer-Verlag, Berlin, Heidelberg, 2013.

\bibitem{AllanCohen2019}
A.~L. Allan and S.~N. Cohen.
\newblock Parameter uncertainty in the {K}alman--{B}ucy filter.
\newblock {\em SIAM J. Control Optim.}, 57:1646--1671, 2019.

\bibitem{AllanCohen2020}
A.~L. Allan and S.~N. Cohen.
\newblock Pathwise stochastic control with applications to robust filtering.
\newblock {\em Ann. Appl. Probab.}, 30(5):2274--2310, 2020.

\bibitem{AllingerMitter1981}
D.~F. Allinger and S.~K. Mitter.
\newblock New results on the innovations problem for non-linear filtering.
\newblock {\em Stochastics}, 4(4):339--348, 1981.

\bibitem{BainCrisan2009}
A.~Bain and D.~Crisan.
\newblock {\em Fundamentals of Stochastic Filtering}.
\newblock Springer, New York, 2009.

\bibitem{BardiDaLio1997}
M.~Bardi and F.~{Da Lio}.
\newblock On the {B}ellman equation for some unbounded control problems.
\newblock {\em Nonlinear Differ. Equ. Appl.}, 4:491--510, 1997.

\bibitem{Borisov2008}
A.~V. Borisov.
\newblock Minimax a posteriori estimation of the {M}arkov processes with finite
  state spaces.
\newblock {\em Autom. Remote Control}, 69:233--246, 2008.

\bibitem{Borisov2011}
A.~V. Borisov.
\newblock The {W}onham filter under uncertainty: a game-theoretic approach.
\newblock {\em Automatica}, 47:1015--1019, 2011.

\bibitem{CaruanaFriz2009}
M.~Caruana and P.~K. Friz.
\newblock Partial differential equations driven by rough paths.
\newblock {\em J. Differential Equations}, 247:140--173, 2009.

\bibitem{CaruanaFrizOberhauser2011}
M.~Caruana, P.~K. Friz, and H.~Oberhauser.
\newblock A (rough) pathwise approach to a class of non-linear stochastic
  partial differential equations.
\newblock {\em Ann. Inst. H. Poincar\'e Anal. Non Lin\'eaire}, 28:27--46, 2011.

\bibitem{ChevyrevFriz2019}
I.~Chevyrev and P.~K. Friz.
\newblock Canonical {RDEs} and general semimartingales as rough paths.
\newblock {\em Ann. Probab.}, 47(1):420--463, 2019.

\bibitem{ChiganskyvanHandel2007}
P.~Chigansky and R.~van Handel.
\newblock Model robustness of finite state nonlinear filtering over the
  infinite time horizon.
\newblock {\em Ann. Appl. Probab.}, 17(2):688--715, 2007.

\bibitem{Cohen2017}
S.~N. Cohen.
\newblock Data-driven nonlinear expectations for statistical uncertainty in
  decisions.
\newblock {\em Electron. J. Stat.}, 11:1858--1889, 2017.

\bibitem{Cohen2020}
S.~N. Cohen.
\newblock Uncertainty and filtering of hidden {M}arkov models in discrete time.
\newblock {\em Probab. Uncertain. Quant. Risk}, 5, Article 4, 2020.

\bibitem{CohenElliott2015}
S.~N. Cohen and R.~J. Elliott.
\newblock {\em Stochastic Calculus and Applications}.
\newblock Springer, New York, 2nd edition, 2015.

\bibitem{CrandallIshiiLions1992}
M.~G. Crandall, H.~Ishii, and P.-L. Lions.
\newblock User's guide to viscosity solutions of second order partial
  differential equations.
\newblock {\em Bull. Amer. Math. Soc.}, 27:1--67, 1992.

\bibitem{CrisanDiehlFrizOberhauser2013}
D.~Crisan, J.~Diehl, P.~K. Friz, and H.~Oberhauser.
\newblock Robust filtering: Correlated noise and multidimensional observation.
\newblock {\em Ann. Appl. Probab.}, 23:2139--2160, 2013.

\bibitem{CrisanRozovskii2011}
D.~Crisan and B.~Rozovskii.
\newblock {\em The Oxford handbook of nonlinear filtering}.
\newblock Oxford University Press, 2011.

\bibitem{DiehlFrizGassiat2017}
J.~Diehl, P.~K. Friz, and P.~Gassiat.
\newblock Stochastic control with rough paths.
\newblock {\em Appl. Math. Optim.}, 75:285--315, 2017.

\bibitem{DiehlFrizMai2016}
J.~Diehl, P.~K. Friz, and H.~Mai.
\newblock Pathwise stability of likelihood estimators for diffusions via rough
  paths.
\newblock {\em Ann. Appl. Probab.}, 26(4):2169--2192, 2016.

\bibitem{FrizHairer2014}
P.~K. Friz and M.~Hairer.
\newblock {\em A Course on Rough Paths, With an Introduction to Regularity
  Structures}.
\newblock Springer, Switzerland, 2014.

\bibitem{FrizOberhauser2014}
P.~K. Friz and H.~Oberhauser.
\newblock Rough path stability of (semi-)linear {SPDEs}.
\newblock {\em Probab. Theory Relat. Fields}, 158:401--434, 2014.

\bibitem{FrizZhang2018}
P.~K. Friz and H.~Zhang.
\newblock Differential equations driven by rough paths with jumps.
\newblock {\em J. Differential Equations}, 264:6226--6301, 2018.

\bibitem{Gubinelli2004}
M.~Gubinelli.
\newblock Controlling rough paths.
\newblock {\em J. Funct. Anal.}, 216:86--140, 2004.

\bibitem{GuoYin2006}
X.~Guo and G.~Yin.
\newblock The {W}onham filter with random parameters: Rate of convergence and
  error bounds.
\newblock {\em IEEE Trans. Automat. Control}, 51(3):460--464, 2006.

\bibitem{Kalman1960}
R.~E. Kalman.
\newblock A new approach to linear filtering and prediction problems.
\newblock {\em J. Basic Eng.}, 82:35--45, 1960.

\bibitem{KalmanBucy1961}
R.~E. Kalman and R.~S. Bucy.
\newblock New results in linear filtering and prediction theory.
\newblock {\em J. Basic Eng.}, 83:95--108, 1961.

\bibitem{MartinMintz1983}
C.~J. Martin and M.~Mintz.
\newblock Robust filtering and prediction for linear systems with uncertain
  dynamics: a game-theoretic approach.
\newblock {\em IEEE Trans. Automat. Contr.}, 28:888--896, 1983.

\bibitem{MillerPankov2005}
G.~Miller and A.~Pankov.
\newblock Filtration of a random process in a statistically uncertain linear
  stochastic differential system.
\newblock {\em Autom. Remote Control}, 66:53--64, 2005.

\bibitem{Siemenikhin2016}
K.~V. Siemenikhin.
\newblock Minimax linear filtering of a random sequence with uncertain
  covariance function.
\newblock {\em Autom. Remote Control}, 77:226--241, 2016.

\bibitem{SiemenikhinLebedevPlatonov2005}
K.~V. Siemenikhin, M.~V. Lebedev, and E.~P. Platonov.
\newblock Kalman filtering by minimax criterion with uncertain noise intensity
  functions.
\newblock {\em Proceedings of the 44\textsuperscript{th} IEEE CDC-ECC}, pages
  1929--1934, 2005.

\bibitem{VerduPoor1984}
S.~Verd\'u and H.~V. Poor.
\newblock Minimax linear observers and regulators for stochastic systems with
  uncertain second-order statistics.
\newblock {\em IEEE Trans. Automat. Contr.}, 29:499--511, 1984.

\bibitem{Wald1945}
A.~Wald.
\newblock Statistical decision functions which minimize the maximum risk.
\newblock {\em Ann. of Math.}, 46:265--280, 1945.

\bibitem{Wonham1965}
W.~M. Wonham.
\newblock Some applications of stochastic differential equations to optimal
  nonlinear filtering.
\newblock {\em J. SIAM Control Ser. A}, 2:347--369, 1965.

\bibitem{YongZhou1999}
J.~Yong and X.~Y. Zhou.
\newblock {\em Stochastic Controls, Hamiltonian Systems and {HJB} Equations}.
\newblock Springer-Verlag, New York, 1999.

\end{thebibliography}

\end{document}